\documentclass[11pt]{article}

\usepackage[english]{babel}
\usepackage[utf8x]{inputenc}
\usepackage[colorinlistoftodos]{todonotes}
\usepackage{float}

\usepackage{amsmath, amssymb, amsthm, cleveref}
\usepackage{graphicx}

\usepackage{mathrsfs}

\usepackage{multicol}
\usepackage{footmisc}

 \newtheorem{theorem}{Theorem}[section]
 \newtheorem{corollary}[theorem]{Corollary}
 
 \newtheorem{lemma}[theorem]{Lemma}
 \newtheorem{proposition}[theorem]{Proposition}

 \theoremstyle{definition}
 \newtheorem{definition}[theorem]{Definition}
 \newtheorem{example}[theorem]{Example}
 \theoremstyle{remark}
 \newtheorem{remark}[theorem]{Remark}

\numberwithin{equation}{section}

 \newcommand{\address}[1]
 { \vspace{-2em}\begin{center}
  \footnotesize{#1}
   \end{center}}

  \newcommand{\email}[1]
   {\vspace{-1em} \begin{center}
   \footnotesize{{\it E-mail address:} \texttt{#1}}
   \end{center}}

  \newcommand{\subjclass}[1]{\footnotesize{{\em 2020 AMS Mathematics
 Subject Classification:} {\bf #1}}}

 \newcommand{\Nat}{\mathbb{N}}
 \newcommand{\Int}{\mathbb{Z}}

 \newcommand{\set}[1]{\left\{#1\right\}}
 \newcommand{\Set}[2]{\set{#1\, :\, #2}}

\newcommand{\Mod}[3]{#1\equiv #2 \pmod{#3}}
\newcommand{\nMod}[3]{#1\not\equiv #2\ (\textrm{mod}\ #3)}
\newcommand{\Div}[2]{#1\,\mid\, #2}
\newcommand{\nDiv}[2]{#1\,\nmid\, #2}

\title{Counting degrees of vertices in near Goldbach graphs}

\author{Shamik Ghosh$^1$\thanks{Corresponding author}\ and Souradeep De$^2$}

\date{}

\begin{document}
\maketitle              

\address{
$^1$Department of Mathematics, Jadavpur University, Kolkata-700032, India\\
\email{ghoshshamik@yahoo.com}
\and
 $^2$Department of Computer Science and Engineering, Jadavpur University, Kolkata-700032, India\\
\email{souradeepde05@gmail.com}
}
%
%
%

\begin{abstract}
A near Goldbach graph is a simple undirected graph whose vertex set consists of all positive even integers and there is an edge between two vertices $a,b$ if and only if $\frac{a+b}{2}, \frac{|a-b|}{2}$ are either odd primes or $1$. A finite near Goldbach graph $G(n)$ has the vertex set $\Set{x\in 2\Nat}{x\leq 2n}$ with the same adjacency rule. In this paper, we obtain two exact formulas for the degree of the even positive integer $x$ in $G(x/2)$. We compute a function $\eta(x)=\prod\limits_{\Div{p}{x},\, p>2} \frac{p-1}{p-2}\, \frac{xe^{-0.183407}}{(\log\, x)^2}$ that approximates the degree of $x$ in $G(x/2)$ for a large even positive integer $x$. Finally, we introduce the concept of a nearly independent set of events and show that if the set of divisibility events for a large even integer $x$ is nearly independent, then $x$ can be expressed as the sum of two odd primes.

\vspace{1em}
\noindent
\end{abstract}



\section{Introduction}

The concept of a Goldbach graph was introduced in 2021, and it is shown that the connectedness of finite Goldbach graphs is equivalent to the famous Goldbach conjecture \cite{DGGS}. In 2025, a slightly modified version of the graph, namely near Goldbach graph, proved to be useful in the study of Goldbach Conjecture \cite{PMMG}. We will see that, if the degree of the highest vertex (even number) $x$ of a finite near Goldbach graph is greater than $1$ then $x$ can be expressed as a sum of two distinct odd primes. This motivates us to count degrees of vertices in near Goldbach graphs. These graphs are lying in a larger class of graphs, called odd-even graphs.

\vspace{1em}
\noindent
A simple undirected graph $G(A,O)=(V,E)$ is an {\em odd even graph} \cite{DGGS} whose vertex set $V=A$ and the edge set $E=\Set{ab}{\frac{a+b}{2}, \frac{|a-b|}{2} \in O,\, a,b\in V}$ for a set $A\subseteq\mathscr{E}=\Set{2n}{n\in\Nat\cup\set{0}}$  and a set $O\subseteq\mathscr{O}=\Set{2n+1}{n\in\Nat\cup\set{0}}$. This graph is necessarily a bipartite graph where partition sets are $X=\Set{x\in V}{\Div{4}{x}}$ and $Y=\Set{x\in V}{\nDiv{4}{x}}$. Interestingly, every bipartite graph $G=G(A,O)$ for some $A\subseteq \mathscr{E}$ and $O\subseteq \mathscr{O}$ \cite{DGGS}. In particular, if $A=\mathscr{E}$ and $O$ consists of all odd prime numbers, then $G(A,O)$ is the {\em (infinite) Goldbach graph} \cite{DGGS}. A {\em finite Goldbach graph} $G_n=G(A,O)$ with $A=\Set{2m}{m\in\Nat\cup\set{0},\, m\leq n}$ and $O=\Set{p}{p\text{ is prime, } 2<p<2n}$ \cite{DGGS}. These graphs are named so, possibly, due to the following theorem:

\begin{theorem}\label{thmgb0} {\em \cite{DGGS}}
An even integer $x\geq 5$ can be expressed as a sum of two odd primes if and only if $G_n$ is connected for all $n\geq 7$.
\end{theorem}

\noindent
The study of the Goldbach graph is difficult as the odd set consists of (odd) prime numbers. It would be much easier to handle same type of graphs where odd sets have a definite pattern. For any natural number $n$, a multiple $mn$ is said to be {\em non-trivial} if $m>1$, $m\in\Nat$. Let $n\in\Nat$ and $p$ be an odd prime. A {\em prime multiple missing graph} \cite{PMMG} $G(p,n)=G(A,O)$ with $A=\Set{x\in 2\Nat}{x\leq 2n}$ and $O$ consists of odd positive integers less than $2n$, which are not non-trivial multiples of $p$. The following is a structure theorem for graphs $G(3,n)$.

\begin{theorem}\label{thmg3n}{\em \cite{PMMG}}
The graph $G(3,n)$ $(n\geq 6)$ is a bipartite graph. The graph consists of an independent set and a path $P$ and vertices of $P$ are alternatively adjacent to all members of the independent set those belong to the opposite partite sets. 
\end{theorem}

\begin{remark}\label{pmm3rem}
Let $n\geq 6$. In the prime multiple missing graph $G=G(3,n)$, the vertex set $V(G)=\Set{2m}{m\in\Nat,\ m\leq n}$. The set $S=\Set{6m}{m\in\Nat}\cap V$ is an independent set. Also the graph is bipartite in which the partition sets are $X=\Set{x\in V}{\Div{4}{x}}$ and $Y=\Set{x\in V}{\nDiv{4}{x}}$. Thus the independent set is further partitioned into two independent sets (in two partite sets), namely, $S_1=\Set{12m+6}{m\in\Nat\cup\set{0}}\cap Y$ and $S_2=\Set{12m}{m\in\Nat}\cap X$. All other even numbers (vertices) form a path $P$ among which multiples of $4$ are adjacent with all members of $S_1$ and other vertices are adjacent to all vertices in $S_2$ (see Figure \ref{gd318}, where bold lines in the figure show adjacency of a path vertex to all vertices in the set $S_1$ or $S_2$).
\end{remark}

\begin{figure}[ht]
\begin{center}
\includegraphics[scale=0.6]{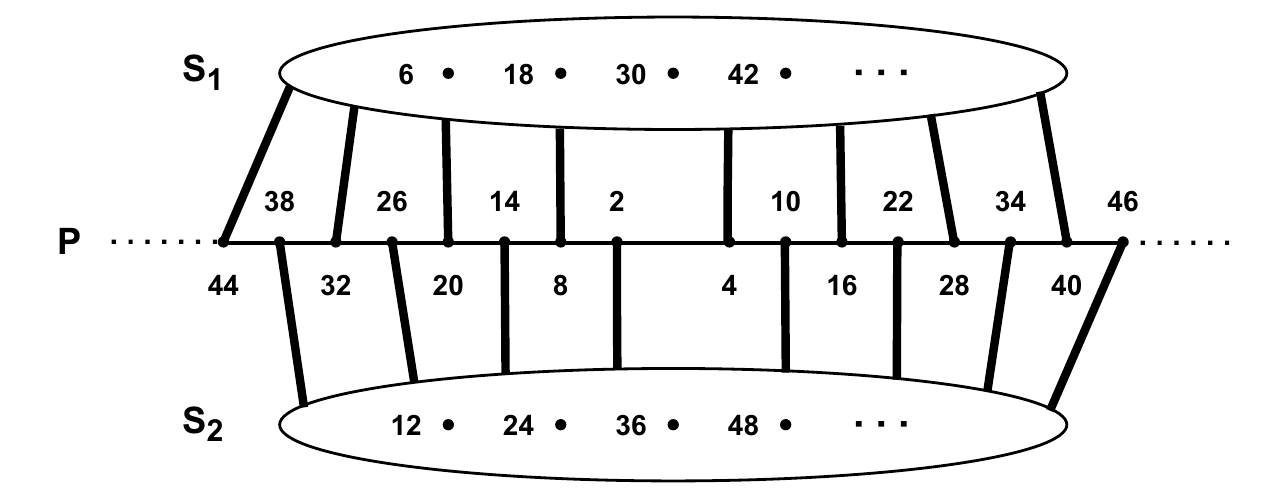}
\caption{A diagram showing prime multiple missing graphs $G(3,n)$}\label{gd318}
\end{center}
\end{figure}

\noindent
Now it is clear that a prime $p$ is not a non-trivial multiple of other primes less than $p$. Thus finite intersection of prime multiple missing graphs gives rise to another class of graphs. For each $n\in\Nat$, a {\em (finite) near Goldbach graph} $G(n)=G(A,O)$ where $A=\Set{y\in 2\Nat}{y\leq 2n}$ and $O=\Set{q}{q\text{ is prime},\ 2<q<2n}\cup\set{1}$ \cite{PMMG}. In the following we write $G(3,5,\ldots,p;n)$ to mean $G(3,n)\cap G(5,n)\cap\cdots\cap G(p,n)$. 

\vspace{1em}
\noindent
By definition, $G(n)=G(3,n)$ for all $n\leq 12$. For $n>12$, we have the following proposition:

\begin{proposition}\label{propngb}{\em \cite{PMMG}}
Let $n\in\Nat$, $n>12$. The near Goldbach graph $G(n)=G(3,5,\ldots,p;n)$ where $p$ is the highest prime such that $p^2<2n$. 
\end{proposition} 

\noindent
Near Goldbach graphs $G(n)$ have some good properties those were missing in (finite) Goldbach graphs, namely that they are not only connected, but also Hamiltonian for small even $n$ and have Hamiltonian paths for all $n\leq 58$ (and also for $n=500$) as shown in \cite{DGGS}. Indeed, following the proof of Theorem \ref{thmgb0} one can obtain the result below. Here we provide a proof for completion.

\begin{theorem}
$G(n)$ is connected for all $n\in\Nat$ if and only if every even integer greater than $2$ can be written as a sum of two distinct positive integers which are odd primes or one.
\end{theorem}

\begin{proof}
Suppose $G(n)$ is connected for all $n\in\Nat$. We have $4=3+1$ and $6=5+1$. Let $x>6$ be an even integer. Let $P=\Set{p}{p\text{ is prime},\ 2<p<x}\cup\set{1}$. Now since $G(x/2)$ is connected, the vertex $x$ is adjacent to some other vertex $y<x$ in $G(x/2)$. Then $p=\frac{x+y}{2},q=\frac{x-y}{2}\in P$. Thus $x=p+q$ where $p$ and $q$ are odd primes or one and $p\neq q$ as $y>0$.

\vspace{1em}
\noindent
Conversely, let every even integer greater than $2$ be a sum of two distinct positive integers which are odd primes or one. Now $G(1)\cong K_1$, $G(2)\cong K_2$, $G(3)\cong P_3$ (a path of length $3$) are connected. Suppose $G(n)$ is connected for some $n\geq 3$. Let $x=2(n+1)$. Then $x=p+q$ for some positive integers $p>q$ which are either odd primes or one. Let $y=p-q$. Then $0<y<x$ which implies $2\leq y\leq 2n$. Now $p=\frac{x+y}{2},q=\frac{x-y}{2}$. Thus $x$ is adjacent to $y$ in $G(n+1)$. Hence $G(n+1)$ is connected.
\end{proof}

\noindent
Let $x$ be an even positive integer. Let us consider the near Goldbach graph $G(x/2)=(V,E)$, where the vertex set $V=V(G(x/2))=\Set{2n}{n\in\Nat,\ n\leq\frac{x}{2}}=\set{2,4,6,\ldots,x}$. Let $y\in V$. Then
$$x\leftrightarrow y \Longleftrightarrow\frac{x+y}{2}=p,\ \frac{x-y}{2}=q\Longleftrightarrow x=p+q,$$
where $p,q$ are odd primes or $1$. Therefore, The degree $d(x)$ of $x$ in the near Goldbach graph $G(x/2)$ counts the number of (unordered) ways that an even positive integer $x$ can be expressed as a sum of two distinct positive integers which are odd primes or $1$. Thus, if the degree of $x$ in $G(x/2)$ is positive then $x$ can be expressed as a sum of two numbers which are either odd primes or $1$ and the degree is greater than $1$ means $x$ can be expressed as a sum of two odd primes. 

\vspace{1em}
\noindent
In this paper, the above fact motivates us to find the degree of $x$ in $G(x/2)$ and to estimate these degrees for large $x$. We found two exact formulas for degrees of even positive integers $x$ in $G(x/2)$ combinatorially, may be for the first time in the literature. In \S 2, first we find an exact formula for degree of $x=12n+4$ ($n\in\Nat$) in the near Goldbach graph $G(x/2)$ by means of combinatorial calculations which is approximated in a simple compact probabilistic form in \S 3. Next we obtain two general formulas for degrees of even positive integers $x$ in $G(x/2)$ for all possible forms $12n+k$ for $k=0,2,4,6,8,10$ in \S 4. Combining all these results, in \S 5, we compute a function $\eta(x)$ in closed form that approximates the degree of a large even positive integer $x$ in $G(x/2)$ using number theoretic tools. We show that $\eta(x)$ is near to a similar function conjectured by Hardy and Littlewood in 1923. In the final section 6, we obtain two sufficient conditions so that an even positive integer can be expressed as a sum of two odd primes. 

\vspace{1em}
\noindent
For literature related to Goldbach conjecture one may go through \cite{chen,HL,RICH,TAO,VIN} and for graph-theoretic concepts and definitions one may consult \cite{DBW}.

\section{Counting degrees of $x=12n+4$ in near Goldbach graphs $G(x/2)$}\label{sec12n4}

In the near Goldbach graph $G(n)$, ($n>12$) the independent set $S_1\cup S_2$ as defined in $G(3,n)$ (see Remark \ref{pmm3rem}) remains the same, but the path $P$ is broken to several pieces and vertices of $P$ are now not adjacent to all vertices in  $S_1$ or $S_2$ as before because the odd set does not contain non-trivial multiples of all odd primes $3\leq p<\sqrt{2n}$. Now by the definition of $G(3,n)$ we get that any neighbor of a vertex (positive even integer) $x$ in $G(n)$ must also be a neighbor of $x$ in $G(3,n)$. In other words, any non-neighbor of $x$ in $G(3,n)$ is also a non-neighbor of $x$ in $G(n)$. This helps us a lot to restrict our search for neighbors within the neighbor sets in $G(3,n)$ only.

\vspace{1em}
\noindent
Let us consider the near Goldbach graph $G(x/2)$ for $x\geq 12$. We will first compute the degree of $x=12n+4$ for some $n\in\Nat$. First of all if $12n+1$ is an odd prime, then $x$ is adjacent to $x-6$ and $x$ can be expressed as a sum of $3$ and an odd prime. This is the only vertex on the path $P$ in $G(3,x/2)$ which is numerically less than $x$ and may be a neighbor of $x$. For example, $16$ is adjacent to $10$ in the near Goldbach graph $G(8)$ as $\frac{16+10}{2}=13$ and $\frac{16-10}{2}=3$ which are prime numbers. But $28$ is not adjacent to $22$ in $G(11)$ as $\frac{28+22}{2}=25$ which is not prime. Since $x-1=12n+3$ is not prime for any $n\in\Nat$, in this case, $x$ can never be expressed as a sum of an odd prime and $1$. Thus all other neighbors are of the form $12m+6$ ($m\in\Nat\cup\set{0}$). 

\vspace{1em}
\noindent
Let $S(n)=\Set{12m+6}{m\in\Nat\cup\set{0}, m<n}$. We wish to count neighbors of $x$ in $S(n)$. First we count in the graph $G(3,5;x/2)$. We will exclude all members $y$ from $S(n)$ such that $5$ divides $x+y$ or $x-y$ (except the case where $x-y=2\times 5=10$). Let us arrange the list of elements in $S(n)$ in increasing order as $\set{6,18,30,42,54,66,78,90,102,114,126,\ldots}=\set{y_1,y_2,y_3,y_4,y_5,y_6,y_7,y_8,y_9,y_{10},y_{11}\ldots}$ (say). These numbers when put to modulo $5$, we get a sequence $\set{1,3,0,2,4\ ,1,3,0,2,4,\ 1,3,0,2,4,\ldots}$. \\
 Now suppose $\Mod{x}{i}{5}$, where $i\in\set{1,2,3,4,5}$. Let $j=5-i$ and it is in $k$-th position in the sequence $\set{1,3,0,2,4}$. Then $x+y_{k+5r}$ is a non-trivial multiple of $5$ for every $r\in\Nat\cup\set{0}$ such that $y_{k+5r}<x$, i.e., $k+5r\leq n$ which implies $r\leq\frac{n-k}{5}$, i.e., $r=0,1,2,\ldots,\lfloor \frac{n-k}{5}\rfloor$. Therefore there are $\lfloor \frac{n-k}{5}\rfloor +1$ number of non-neighbors $y$ of $x$ in the graph $G(3,5;x/2)$ within $S(n)$ for which $\Div{5}{x+y}$. Let us denote the set of these non-neighbors by $F(x,5,+)$.

\vspace{1em}
\noindent
For example, let $n=8$. Then $x=100$, $S(8)=\set{6,18,30,42,54,66,78,90}$. Now $\Mod{x}{5}{5}$. So $j=5-5=0$ and $k=3$. Thus we have to exclude members of $S(8)$ in the positions $\set{3,8}$. Note that $\lfloor \frac{n-k}{5}\rfloor +1=2$ and $F(100,5,+)=\set{30,90}$. Again let $n=10$. Then $x=124$, $S(10)=\set{6,18,30,42,54,66,78,90,102,114}$, $\Mod{x}{4}{5}$, $j=5-4=1$, $k=1$. So members of $S(10)$ which are in positions $\set{1,6}$ are non-neighbors of $x$. Here also $\lfloor \frac{n-k}{5}\rfloor +1=2$ and $F(124,5,+)=\set{6,66}$. 

\vspace{1em}
\noindent
Similarly, when $\Mod{x}{i}{5}$, where $i\in\set{0,1,2,3,4}$. Suppose $i$ is in the $k$-th position in the sequence $\set{1,3,0,2,4}$. Then $\Div{5}{x-y_{k+5r}}$ for every $r\in\Nat\cup\set{0}$ such that $y_{k+5r}<x$. Among them there is some $y=12t+6$ ($t\in\Nat\cup\set{0}$) such that $\frac{x-y}{2}=5$. Then $12n+4-(12t+6)=2\times 5$ which implies $t=\frac{6n-5-1}{6}$. So its position in $S(n)$ is $\frac{6n-5-1}{6}+1=\frac{6n-5+5}{6}=n$. Thus we consider elements $y_{k+5r}$ for $r=0,1,2,\ldots,\lfloor \frac{n-k}{5}\rfloor$, except $y_n$. We denote the set of these non-neighbors by $F(x,5,-)$. We have $F(100,5,-)=\set{30}$ and $F(124,5,-)=\set{54}$. Thus $100$ has only $2$ non-neighbors $\set{30,90}$ in $S(8)$ whereas $124$ has $3$ non-neighbors $\set{6,54,66}$ in $S(10)$. In general, there are exactly $|F(x,5,+)\cup F(x,5,-)|$ non-neighbors of $x$ within $S(n)$ in the graph $G(3,5;x/2)$. In other words, there are $n-|F(x,5,+)\cup F(x,5,-)|$ neighbors of $x$ within $S(n)$ in the graph $G(3,5;x/2)$.

\begin{definition}\label{deffp}
Let $S(n)=\Set{12m+6}{m\in\Nat\cup\set{0}, m<n}$. Let $p$ be a prime number such that $3<p<\sqrt{x}$. Let 
$$F(x,p,+)=\{y\in S(n)\, :\, \Div{p}{x+y}\} \text{ and } F(x,p,-)=\{y\in S(n)\, :\, \Div{p}{x-y},\ x-y\neq 2p\}.$$ 
\end{definition}

\begin{lemma}\label{lem12n4}
Let $p$ be a prime number such that $3<p<\sqrt{x}$. Let $\Mod{x}{i}{p}$, $i\in\Nat$, $i\leq p$. Let $j=p-i$ and it is in the $k_1$-th position in the sequence $\Set{[12t+6]_p}{t\in\Nat\cup\set{0},\ t\leq p-1}$, where $[u]_p$ denotes the number $c\in\Nat\cup\set{0}$, $c\leq p-1$ such that $\Mod{u}{c}{p}$. Then 
$$F(x,p,+)=\Set{y_{k_1+pr}}{r=0,1,2,\ldots,\lfloor \frac{n-k_1}{p}\rfloor}.$$
Similarly, if $\Mod{x}{i}{p}$, $i\in\Nat\cup\set{0}$, $i\leq p-1$ and $i$ is in the $k_2$-th position in the sequence $\Set{[12t+6]_p}{t\in\Nat\cup\set{0},\ t\leq p-1}$, then 
$$F(x,p,-)=\Set{y_{k_2+pr}}{r=0,1,2,\ldots,\lfloor \frac{n-k_2}{p}\rfloor} \setminus\Set{y_s}{s=\frac{6n-p+5}{6},\ s\in\Nat\cup\set{0}}.$$
Moreover, the set $\set{y_s}\neq\emptyset$ if and only if $\Div{3}{p+1}$.
\end{lemma}

\begin{proof} 
Let $x=12n+4$ for some $n\in\Nat$ and $p$ be a prime number such that $3<p<\sqrt{x}$. Let $y\in F(x,p,+)$. As we noticed $y\in S(n)=\Set{12m+6}{m\in\Nat\cup\set{0}, m<n}$. Let us arrange members of $S(n)$ in increasing order and apply modulo $p$ for each of them. Then we get a repeating sequence of $p$ numbers. For example, for $p=5$, we got the sequence $\set{1,3,0,2,4\ ,1,3,0,2,4,\ 1,3,0,2,4,\ldots}$. Since we are trying to locate elements $y\in S(n)$ for which $p$ divides $x+y$, we choose those $y$ such that $\Mod{y}{j}{p}$, where $j=p-i$ and $\Mod{x}{i}{p}$. We will get a unique such $y$ in the first block of $p$ numbers in the above repeating sequence which is located at say, $k_1$-th position in the block $\Set{[12t+6]_p}{t\in\Nat\cup\set{0},\ t\leq p-1}$. Also in each subsequent block of $p$ numbers we find a unique such $y$ at the same position of the block. Thus these numbers are $\set{y_{k_1},y_{k_1+p},y_{k_1+2p},\ldots,y_{k_1+pt}}$, where $k_1+pt\leq n$, i.e., $t\leq \lfloor \frac{n-k_1}{p}\rfloor$. Hence we have
$$F(x,p,+)=\Set{y_{k_1+pr}}{r=0,1,2,\ldots,\lfloor \frac{n-k_1}{p}\rfloor}.$$

\noindent
Similarly, $F(x,p,-)$ consists of those $y$ in $S(n)$ such that $p$ divides $x-y$, except possibly the case where $p=\frac{x-y}{2}$. Now $p$ divides $x-y$ implies that $\Mod{y}{i}{p}$, where $\Mod{x}{i}{p}$ and let the number $i$ be in the $k_2$-th position within the first block of $p$ numbers in the repeating sequence. So in this case these numbers are precisely $\set{y_{k_2},y_{k_2+p},y_{k_2+2p},\ldots,y_{k_2+pt}}$, where $t\leq \lfloor \frac{n-k_2}{p}\rfloor$ except (possibly) the case $y=12r+6$ for which $p=\frac{x-y}{2}=\frac{(12n+4)-(12r+6)}{2}=6n-6r-1$ which implies $r=\frac{6n-p-1}{6}$. Since we have started with counting $y_1$ for $r=0$, this element lies in the $s=(r+1)$-th position, i.e., $s=\frac{6n-p-1}{6}+1=\frac{6n+5-p}{6}$. So we have to exclude $y_s$ whenever $s$ is an integer. Thus we have
$$F(x,p,-)=\Set{y_{k_2+pr}}{r=0,1,2,\ldots,\lfloor \frac{n-k_2}{p}\rfloor} \setminus\Set{y_s}{s=\frac{6n-p+5}{6},\ s\in\Nat\cup\set{0}}.$$
Note that $s$ is an integer if and only if $6$ divides $6n+5-p=6(n+1)-(p+1)$, i.e., if and only if $6$ divides $p+1$. Since $p+1$ is always even, we have $s$ is an integer if and only if $3$ divides $p+1$. 
\end{proof}

\noindent
Now in the graph $G(x/2)$, the number of non-neighbors of $x=12n+4$ in $S(n)$ is given by
\begin{equation}\label{eqfn}
f(n)=|\bigcup\limits_{\substack{p\, =\, \text{prime}\\ 3<p<\sqrt{x}}} \ \left(F(x,p,+)\cup F(x,p,-)\right)|.
\end{equation}
Therefore we have the following theorem.

\begin{theorem}\label{thm14}
Let $n\in\Nat$. Then the degree of $x=12n+4$ in $G(x/2)$ is $n-f(n)+1$ or $n-f(n)$ according to $x-3$ is prime or not, where $f(n)$ is defined by (\ref{eqfn}). 
\end{theorem}

\begin{proof}
The proof follows from Lemma \ref{lem12n4} and Proposition \ref{propngb}.
\end{proof}

\noindent
If $x-3$ is prime, then $x$ is a sum of two odd primes. If $x-3$ is not a prime number, $x$ can be expressed as a sum of two odd primes if and only if $n-f(n)>0$, i.e., $f(n)<n \Longleftrightarrow \frac{f(n)}{n}<1$. Since $12n+3$ cannot be prime, $x$ cannot be expressed as the sum of an odd prime and $1$. Now let us count $|F(x,p,+)|$, $|F(x,p,-)|$ and $|F(x,p,+)\cup F(x,p,-)|$ precisely.

\begin{lemma}\label{lemfp4}
$|F(x,p,+)|=\lfloor (n-k_1)/p\rfloor+1$, $|F(x,p,-)|=\lfloor (n-k_2)/p\rfloor $ if\ $\Div{3}{p+1}$ and $|F(x,p,-)|=\lfloor (n-k_2)/p\rfloor +1$ if\ $\nDiv{3}{p+1}$.  
$$|F(x,p,+)\cup F(x,p,-)|=\left\{\begin{array}{ll}
\lfloor (n-k_1)/p\rfloor+1, & \text{if } \Div{p}{x}\\
\lfloor (n-k_1)/p\rfloor + \lfloor (n-k_2)/p\rfloor +1, & \text{if } \nDiv{p}{x}\text{ and } \Div{3}{p+1}\\
\lfloor (n-k_1)/p\rfloor + \lfloor (n-k_2)/p\rfloor +2, & \text{if } \nDiv{p}{x}\text{ and } \nDiv{3}{p+1},\\
\end{array}\right.$$
where $k_1=[(p-[x]_p^1-6)\times 12^{p-2}]_p+1$, $k_2=[(x-6)\times 12^{p-2}]_p+1$, $[a]_p=i\in\set{0,1,2,\ldots,p-1}$ and $[a]_p^1=i\in\set{1,2,\ldots,p}$ such that $\Mod{a}{i}{p}$.
\end{lemma}

\begin{proof}
We have $F(x,p,+)=\Set{y_{k_1+pr}}{r=0,1,2,\ldots,\lfloor \frac{n-k_1}{p}\rfloor}$, where $\Mod{x}{i}{p}$, $i\in\Nat$, \\
$i\leq p$, $j=p-i$ and it is in the $k_1$-th position in the sequence $\Set{[12t+6]_p}{t\in\Nat\cup\set{0},\ t\leq p-1}$. Suppose $[12t+6]_p=[p-[x]_p^1]_p$. Then $t=[p-[x]_p^1-6]_p\, ([12]_p)^{-1}$. By Fermat's Little Theorem, we have $\Mod{12^{p-1}}{1}{p}$ as $p>3$ is a prime. Thus $([12]_p)^{-1}=[12^{p-2}]_p$. This implies $t=[(p-[x]_p^1-6)\times 12^{p-2}]_p$. Finally, since $t$ starts from $0$, we have $k_1=[(p-[x]_p^1-6)\times 12^{p-2}]_p+1$. Then $|F(x,p,+)|=\lfloor (n-k_1)/p\rfloor+1$. 

\vspace{1em}
\noindent
Again $F(x,p,-)=\Set{y_{k_2+pr}}{r=0,1,2,\ldots,\lfloor \frac{n-k_2}{p}\rfloor} \setminus\Set{y_s}{s=\frac{6n-p+5}{6},\ s\in\Nat\cup\set{0}}$,\\
 where $\Mod{x}{i}{p}$, $i\in\Nat\cup\set{0}$, $i\leq p-1$ and $i$ is in the $k_2$-th position in the sequence\\
 $\Set{[12t+6]_p}{t\in\Nat\cup\set{0},\ t\leq p-1}$. Now $[12t+6]_p=[x]_p$ implies that $t=[(x-6)\times 12^{p-2}]_p$ and hence $k_2=[(x-6)\times 12^{p-2}]_p+1$. Thus $|F(x,p,-)|$ is $\lfloor (n-k_2)/p\rfloor$ or $\lfloor (n-k_2)/p\rfloor +1$ according to $s=\frac{6n-p+5}{6}$ is an integer or not, i.e., $\Div{6}{p+1}$ or not which is equivalent to say $\Div{3}{p+1}$ or not as $p+1$ is even.

\vspace{1em}
\noindent
Next we notice that, by definition of $F(x,p,+)$ and $F(x,p,-)$ the cases are exactly same (with possible exception of one case where $\frac{x-y}{2}=p$) when $p$ divides $x$, i.e., when $\Mod{x}{0}{p}$. In this case the position of $j$ for $F(x,p,+)$ and the position of $i$ in $F(x,p,-)$ coincide. Otherwise all the cases are different. Thus, we have $|F(x,p,+)\cup F(x,p,-)|=|F(x,p,+)|$ or $|F(x,p,+)|+|F(x,p,-)|$ according to $\Div{p}{x}$ or not. This proves the formula in the statement. 
\end{proof}

\begin{example}
Let $n=100$, $x=1204$, $p=5$. Then $\Mod{x}{4}{5}$ and $j=5-4=1$. So $k=1$ as $1$ is in the first position of the modulo sequence $\set{1,3,0,2,4}$ for $5$. So there are $20$ occurrences of $y=12m-6$ for $m=1,6,11,16,21,26,31,36,41,46,51,56,61,66,71,76,81,86,91,96$ for which $\Div{5}{x+y}$. Again $4$ is in the $5$-th position of $\set{1,3,0,2,4}$. Hence there are $19$ occurrences of $y=12m-6$ for $m=5,10,15,20,25,30,35,40,45,50,55,60,65,70,75,80,85,90,95$ for which $\Div{5}{x-y}$. All these cases are different. Thus here $|F(x,p,+)\cup F(x,p,-)|=|F(x,p,+)|+|F(x,p,-)|$. 

\vspace{1em}
\noindent
Let $p=7$. Here $\Div{7}{x}$, i.e., $\Mod{x}{7}{7}$ and $j=7-7=0$. Now the modulo sequence of $7$ is $\set{6,4,2,0,5,3,1}$ where $0$ is in the $4$-th position. Thus there are $14$ occurrences of $y=12m-6$ for $m=4,11,18,25,32,39,46,53,60,67,74,81,88,95$ for which $\Div{7}{x+y}$. Again since $0$ is in the $4$-th position in $\set{6,4,2,0,5,3,1}$, there are same $14$ occurrences of $y=12m-6$ for which $\Div{7}{x-y}$. All these cases are same. Thus here $|F(x,p,+)\cup F(x,p,-)|=|F(x,p,+)|$.
\end{example}

\section{Some Approximations}

\noindent
In Theorem \ref{thm14}, we got an exact formula for computing degree of an even integer $x=12n+4$ in $G(x/2)$ for all $n\in\Nat$. This formula is difficult to compute for large $n$ as it involves union of sets. In the following, we approach a probabilistic argument to estimate the degrees for large $n$. 

\begin{definition}\label{defap}
Let $n\in\Nat$, $x=12n+4$, $S(n)=\Set{12m+6}{m\in\Nat\cup\set{0},\ m<n}$.\\
 Let $PR(x)$ be the set of primes $p$ such that $3<p<\sqrt{x}$ and $r_x=|PR(x)|$. For each prime $p\in PR(x)$, let $A_{p}$ be the event which is the set of numbers $y$ in $S(n)$ such that either $\Div{p}{x+y}$ or $\Div{p}{x-y}$ (except the case where $p=\frac{x-y}{2}$). Let $\mathscr{A}(x)=\Set{A_p}{p\in PR(x)}$ and $\mathscr{A}_x=\bigcup \Set{A_p}{p\in PR(x)}$. We call $\mathscr{A}(x)$ as the {\em set of divisibility events} for $x$. Note that for $n>3$, $|\mathscr{A}(x)|=|PR(x)|>1$. Moreover, we define $T(n)=\frac{f(n)}{n}$, where $f(n)$ is given by (\ref{eqfn}).
\end{definition}

\noindent
It is clear that $T(n)$ is the probability of an element $12m+6$ in $S(n)$ such that $\frac{x+y}{2}$ or $\frac{x-y}{2}$ is a non-trivial multiple of a prime $3<p<\sqrt{x}$, i.e., the probability of the event $\mathscr{A}_x$. Now in view of Lemma \ref{lemfp4}, we define the following:
	
\begin{definition}\label{defeta13}
Let $n\in\Nat$ and $x=12n+4$. For any prime $p>3$, we define
$$pr_1(p)=\left\{\begin{array}{ll}
\frac{\lfloor (n-k_1)/p\rfloor+1}{n}, & \text{if } \Div{p}{x}\\
\frac{\lfloor (n-k_1)/p\rfloor + \lfloor (n-k_2)/p\rfloor +1}{n}, & \text{if } \nDiv{p}{x}\text{ and } \Div{3}{p+1}\\
\frac{\lfloor (n-k_1)/p\rfloor + \lfloor (n-k_2)/p\rfloor +2}{n}, & \text{if } \nDiv{p}{x}\text{ and } \nDiv{3}{p+1},\\
\end{array}\right.$$
where $k_1$ and $k_2$ are defined as in Lemma \ref{lemfp4} and

$$pr_2 (p)=\left\{\begin{array}{ll}
\frac{1}{p}+\frac{1}{n}, & \text{if } \Div{p}{x}\\
\frac{2}{p}+\frac{1}{n}, & \text{if } \nDiv{p}{x}\text{ and } \Div{3}{p+1}\\
\frac{2}{p}+\frac{2}{n}, & \text{if } \nDiv{p}{x}\text{ and } \nDiv{3}{p+1}.\\
\end{array}\right.$$
Now we define the following approximations of $T(n)$:

\begin{equation}\label{eqneta5n}
\tau_1(n)=1-\prod\limits_{i=1}^{r_x} (1-pr_1(p_i)).
\end{equation}

\begin{equation}\label{eqneta6n}
\tau_2(n)=1-\prod\limits_{i=1}^{r_x} (1-pr_2 (p_i)).
\end{equation}
\end{definition}

\noindent
In Table \ref{tabtnetan}, we compare the values of $T(n)$ and two approximate functions defined above by assuming events $A_p$ ($p>3$) are mutually independent. We observe that eventually the values of $\tau_1(n)$ and $\tau_2(n)$ are same as $n$ becomes large and the actual values are nearly $0.2\%$ more than these approximations for $n=10^{10}$, i.e., for $x=12\times 10^{10}+4$. The actual values of $T(n)$ are evaluated by computing (\ref{eqfn}). 

\begin{table}[t]
$$\begin{array}{|l|l|l|l|l|l|}
\hline
n & T(n) & \tau_1(n) & \tau_2(n) & T(n)/\tau_2(n) & \text{deg}(12n+4)\\
\hline 
10 &  0.5 & 0.559 & 0.815325 &  0.61325 & 5\\
\hline
10^2 & 0.73 & 0.761912 & 0.806097 &  0.9056 & 28\\
\hline
10^3 & 0.888 & 0.891806 & 0.898164 &  0.98868 & 112\\
\hline
10^4 & 0.9251 & 0.922967 & 0.924014 & 1.00118 & 749 \\
\hline
10^5 & 0.94666 & 0.943553 & 0.943756 &  1.00308 & 5334\\
\hline
10^6 & 0.965814 & 0.962464 & 0.962501 &  1.00344 & 34186\\
\hline
10^7 & 0.974293 & 0.971193 & 0.971201 &  1.00318 & 257074\\
\hline 
10^8 & 0.975922 & 0.972605 & 0.972607 &  1.00341 & 2407769\\
\hline
10^9 & 0.983884 & 0.981462 & 0.981462 &  1.00247 & 16115802\\
\hline
10^{10} & 0.986781 & 0.984656 & 0.984656 &  1.00216 & 132188594\\
\hline
\end{array}$$
\caption{Exact degrees of $x$ in $G(x/2)$ and comparison between values of $T(n)$ and $\tau_1(n),\tau_2(n)$ for $x=12n+4$.}\label{tabtnetan}
\end{table}

\vspace{1em}
\noindent
Now by Theorem \ref{thm14} the degree of $x$ in $G(x/2)$ is $n(1-T(n))+1$ or $n(1-T(n))$ according as $x-3$ is prime or not. The final computation of degree of $x$ depends on the fact that whether the cases for different primes are independent or not. Though there is no such apparent dependency between these cases, in reality, they are not independent. In the following we wish to study the dependence of events $A_p$ ($p>3$) for two different primes. We begin with estimating the probability, $P(A_p\cap A_q)$ for some different primes $p,q\geq 5$.

\begin{lemma}\label{lemindfp}
Let $p,q$ be two prime numbers such that $p\neq q$ and $p,q\geq 5$. Let $k_p=[(p-[x]_p^1-6)\times 12^{p-2}]_p+1$, $k_q=[(q-[x]_q^1-6)\times 12^{q-2}]_q+1$, where $[a]_p=i\in\set{0,1,2,\ldots,p-1}$ and $[a]_p^1=i\in\set{1,2,\ldots,p}$ such that $\Mod{a}{i}{p}$. Then 
$$|F(x,p,+)\cap F(x,q,+)|= \min\left\{\frac{\lfloor \frac{n-k_p}{p}\rfloor -ca_0}{q}, \frac{\lfloor \frac{n-k_q}{q}\rfloor -cb_0}{p}\right\} + \max\, \left\{-\frac{ca_0}{q},-\frac{cb_0}{p}\right\} +1,$$
where $a_0,b_0$ be a solution of the equation $pa-qb=1$ and $c=k_q-k_p$.
\end{lemma}

\begin{proof}
We have, by Lemma \ref{lem12n4}, \\
$F(x,p,+)=\Set{y_{k_p+pr}}{r=0,1,2,\ldots,\lfloor \frac{n-k_p}{p}\rfloor}$, where $k_p=[(p-[x]_p^1-6)\times 12^{p-2}]_p+1$. A similar formula holds for $q$. Now to find the intersection $F(x,p+)\cap F(x,q,+)$ we solve for $a,b$ from the following diophantine equation: 
$$k_p+pa=k_q+qb,\text{ i.e., }pa-qb=k_q-k_p=c \text{ (say)}.$$
Let $a_0,b_0$ be a solution of the equation $pa-qb=1$, then the general solution is $a=ca_0+qt$ and $b=cb_0+pt$, where $t$ is any integer. Now in the expression of $F(x,p,+)$, we have $0\leq a\leq \lfloor \frac{n-k_p}{p}\rfloor$ which implies $-\frac{ca_0}{q}\leq t\leq \frac{\lfloor \frac{n-k_p}{p}\rfloor -ca_0}{q}$. Similarly, we have $0\leq b\leq \lfloor \frac{n-k_q}{q}\rfloor\Longrightarrow -\frac{cb_0}{p}\leq t\leq \frac{\lfloor \frac{n-k_q}{q}\rfloor -cb_0}{p}$. Thus we have 
$$\max\, \left\{-\frac{ca_0}{q},-\frac{cb_0}{p}\right\}\leq t\leq \min\left\{\frac{\lfloor \frac{n-k_p}{p}\rfloor -ca_0}{q}, \frac{\lfloor \frac{n-k_q}{q}\rfloor -cb_0}{p}\right\}.$$
Hence 
$$|F(x,p,+)\cap F(x,q,+)|= \min\left\{\frac{\lfloor \frac{n-k_p}{p}\rfloor -ca_0}{q}, \frac{\lfloor \frac{n-k_q}{q}\rfloor -cb_0}{p}\right\} + \max\, \left\{-\frac{ca_0}{q},-\frac{cb_0}{p}\right\} +1.$$
\end{proof}

\noindent
Now in order to calculate degree of $x$ in $G(x/2)$ we are concerned with primes less than $\sqrt{x}$. We note that $n=\frac{x-4}{12}>\sqrt{x}$ for $n\geq 152$. Now if we consider $n$ to be large enough in comparison with primes, $p,q$, then we may approximate $|F(x,p,+)\cap F(x,q,+)|\approx \frac{n}{pq}$ (where we may ignore the other small constants).

\vspace{1em}
\noindent
Moreover we have seen in the proof of Lemma \ref{lemfp4} that either $F(x,p,+)\cap F(x,p,-)=\emptyset$ or $F(x,p,+)\cap F(x,p,-)=F(x,p,-)$ according to $\Div{p}{x}$ or not. Now what we calculate for $F(x,p,+)$ and $F(x,q,+)$, similar expressions can be obtained for $F(x,p,+)$ and $F(x,q,-)$ or $F(x,p,-)$ and $F(x,q,+)$ or $F(x,p,-)$ and $F(x,q,-)$. Thus in general we may estimate the following:

\begin{corollary}
Let $p,q$ be two distinct primes greater than $3$. The events $A_p$ and $A_q$ are defined as in Definition \ref{defap}. Then
$$P(A_p\cap A_q)=(|(F(x,p,+)\cup F(x,p,-))\cap (F(x,q,+)\cup F(x,q,-))|)/n \approx$$
$$\left\{%
\begin{array}{ll}
\frac{1}{pq} & \text{if } \Div{p}{x}\text{ and }\Div{q}{x}\\
\frac{2}{pq} & \text{if } \Div{p}{x}\text{ and }\nDiv{q}{x}\\
\frac{2}{pq} & \text{if } \nDiv{p}{x}\text{ and }\Div{q}{x}\\
\frac{4}{pq} & \text{if } \nDiv{p}{x}\text{ and }\nDiv{q}{x}.
\end{array}\right.$$
\end{corollary}

\begin{proof}
The proof follows immediately from Lemma \ref{lemindfp} and the above discussion.
\end{proof}

\begin{example}
Suppose $n=1000$, $x=12n+4$. Let $p=5$ and $q=7$. We have \\
$k_5=[(5-[x]_5^1-6)\times 12^{5-2}]_5+1=1$ and $k_7=[(7-[x]_7^1-6)\times 12^{7-2}]_7+1=7$. \\
Also $\lfloor \frac{1000-k_5}{5}\rfloor=\lfloor 199.8\rfloor = 199$ and $\lfloor \frac{1000-k_7}{7}\rfloor=\lfloor 141.86\rfloor = 141$. Then we have \\
$F(x,5,+)=\Set{1+5j}{j\in\Int,\ 0\leq j\leq 199}$ and $F(x,7,+)=\Set{7+7i}{i\in\Int,\ 0\leq i\leq 141}$. Now a solution for the diophantine equation $5a-7b=1$ is $a_0=3$, $b_0=2$. Also $c=k_7-k_5=6$. Thus the general solutions are $a=18+7t$ and $b=12+5t$, $t\in\Int$. Thus the required range for $t$ is given by
$\max\, \left\{-\frac{ca_0}{q},-\frac{cb_0}{p}\right\}\leq t\leq \min\left\{\frac{\lfloor \frac{n-k_p}{p}\rfloor -ca_0}{q}, \frac{\lfloor \frac{n-k_q}{q}\rfloor -cb_0}{p}\right\}$ 
$ \Longrightarrow\ \max\, \left\{-\frac{18}{7},-\frac{12}{5}\right\}\leq t\leq \min\left\{\frac{199 -18}{7}, \frac{141 -12}{5}\right\}\ \Longrightarrow\ \ -2.4\leq t \leq 25.8$.
Since $t\in\Int$, we have $-2\leq t\leq 25$. Thus $|F(x,5,+)\cap F(x,7,+)|=25+2+1=28$. In fact 
$F(x,5,+)\cap F(x,7,+)=\{21, 56, 91, 126, 161, 196,$\\ 
$231, 266, 301, 336, 371, 406, 441, 476, 511, 546, 581, 616, 651, 686, 721, 756, 791, 826, 861, 896, 931, 966\}$.

\vspace{1em}
\noindent
Similarly, we can show that $|F(x,5,-)\cap F(x,7,+)|=28$, $|F(x,5,+)\cap F(x,7,-)|=29$, $|F(x,5,-)\cap F(x,7,-)|=29$. Thus 
$|(F(x,5,+)\cup F(x,5,-))\cap (F(x,7,+)\cup F(x,7,-))|=114$ as $F(x,5,+)\cap F(x,5,-)=\emptyset$ and $F(x,7,+)\cap F(x,7,-)=\emptyset$ (for $\nDiv{5}{x}$ and $\nDiv{7}{x}$). Thus $P(A_5\cap A_7)=\frac{114}{1000}=0.114$. Note that the approximate value is $\frac{4}{35}=0.114286$. Also note that the approximate value of $|F(x,5,+)\cap F(x,7,+)|=\frac{1000}{35}=28.5714$.
\end{example}

\noindent
Therefore, for large $n$, if we approximate for any prime $p>3$,
$$P(A_p)=(|F(x,p,+)\cup F(x,p,-)|)/n\approx\left\{%
\begin{array}{ll}
\frac{1}{p} & \text{if } \Div{p}{x}\\
\frac{2}{p} & \text{if } \nDiv{p}{x}.\\
\end{array}\right.$$
Then we have $P(A_p\cap A_q)=P(A_p) P(A_q)$ which implies the events are (approximately) independent. Thus we may approximate $T(n)$ by the following function $\tau_3(n)$, where $pr_3(p)$ is the above approximation of $P(A_p)$.
\begin{equation}\label{eqnetaap}
\tau_3(n)=1-\prod\limits_{i=1}^{r_x} (1-pr_3(p_i)).
\end{equation}

\begin{table}[hb]
{\footnotesize $$\begin{array}{|c|l|l|l|l|l|l|l|l|l|l|}
\hline
n & 10 & 100 & 1000 & 10000 & 10^5 & 10^6 & 10^7 & 10^8 & 10^9 & 10^{10}\\
\hline
T(n) & 0.5 & 0.73 & 0.888 & 0.9251 & 0.94666 & 0.965814 & 0.974293 & 0.975922 & 0.983884 & 0.986781\\
\hline
\tau_3(n) & 0.649351 & 0.776471 & 0.893623 & 0.923243 & 0.943602 & 0.962473 & 0.971195 & 0.972605 & 0.981462 & 0.984656\\
\hline
\end{array}$$}
\caption{Comparison table for $T(n)$ and $\tau_3(n)$ for $x=12n+4$.}\label{tabtnetaapn}
\end{table}

\vspace{1em}
\noindent
Table \ref{tabtnetaapn} compares the actual value of $T(n)$ with the new approximation function $\tau_3(n)$. We again observe that the values of $\tau_3(n)$ merges with $\tau_1(n)$ or $\tau_2(n)$ for large $n$. We return back to further approximation to get a more compact closed form in \S 5.

\section{Counting degrees of $x=12n+k$ for $k=0,2,6,8,10$ in G(x/2)}\label{seckother}

In this section, we count degrees of even integers $x=12n+k$ ($n\in\Nat$) where $k=0,2,6,8,10$. 

\begin{theorem}\label{thm24810}
Let $n\in\Nat$. Then the degree of $x=12n+k$, $(k=2,4,8,10)$ in $G(x/2)$ is $n_1+1-f(n)$ or $n_1-f(n)$, according as $x-3$ is prime or not, where
$$f(n)=|\bigcup\limits_{\substack{p\, =\, \text{prime}\\ 3< p<\sqrt{x}}} \ \left(F(x,p,+)\cup F(x,p,-)\right)|,$$
$$|F(x,p,+)\cup F(x,p,-)|=\left\{\begin{array}{ll}
\lfloor (n_1-k_1)/p\rfloor+1, & \text{if } \Div{p}{x}\\
\lfloor (n_1-k_1)/p\rfloor + \lfloor (n_1-k_2)/p\rfloor +1, & \text{if } \nDiv{p}{x}\text{ and } \Div{3}{p+c}\\
\lfloor (n_1-k_1)/p\rfloor + \lfloor (n_1-k_2)/p\rfloor +2, & \text{if } \nDiv{p}{x}\text{ and } \nDiv{3}{p+c},\\
\end{array}\right.$$
$k_1=[(p-[x]_p^1-b)\times 12^{p-2}]_p+1$ and $k_2=[(x-b)\times 12^{p-2}]_p+1$, where $[a]_p=i\in\set{0,1,2,\ldots,p-1}$ and $[a]_p^1=i\in\set{1,2,\ldots,p}$ such that $\Mod{a}{i}{p}$, $c=1$ for $k=4,10$ and $c=-1$ for $k=2,8$, $b=6$ for $k=4,8$ and $b=12$ for $k=2,10$, $n_1=n+1$ for $k=8$ and $n_1=n$ for $k=2,4,10$.
\end{theorem}

\begin{proof}
First we consider $x=12n+8$, $n\in\Nat$. If $12n+5$ is prime, then $x$ can be expressed as the sum of an odd prime plus $3$. In all other cases $x$ is adjacent to only even numbers of the form $12m+6$, $m\in\Nat\cup\set{0}$, $m\leq n$. This includes the case when $12n+7$ is prime and $x$ can be expressed as the sum of an odd prime plus $1$. Thus as before we have expressions and formulas for $F(x,p,+)$ and $F(x,p,-)$ similar to the case of $x=12n+4$ with the exception that all instances of $n$ in the formulas of Lemma \ref{lemfp4} are to be replaced by $n_1=n+1$. In this case $s=\frac{12n+8-(2p-12)-6}{12}=\frac{6n-p+7}{6}$.

\vspace{1em}
\noindent
Next we consider $x=12n+2$, $n\in\Nat$. If $12n-1$ is prime, then $x$ can be expressed as the sum of an odd prime plus $3$. In all other cases $x$ is adjacent to only even numbers of the form $12m$ ($m\in\Nat$, $m\leq n$). This includes the case when $12n+1$ is prime and $x$ can be expressed as the sum of an odd prime plus $1$. Thus $k_1=m+1$, where $[12m+12]_p=[p-[x]_p^1]_p$ which implies $k_1=[(p-[x]_p^1-12)\times 12^{p-2}]_p+1$ and $k_2=[(x-12)\times 12^{p-2}]_p+1$. In this case $s=\frac{12n+2-(2p-12)-12}{12}=\frac{6n-p+1}{6}$.

\vspace{1em}
\noindent
Let $x=12n+10$, $n\in\Nat$. Since $12n+9$ cannot be a prime number, $x$ cannot be expressed as the sum of an odd prime and $1$. But if $12n+7$ is prime, then $x$ can be expressed as the sum of an odd prime plus $3$. In all other cases $x$ is adjacent to only even numbers of the form $12m$ ($m\in\Nat$). So we have same formulas for $k_1$ and $k_2$ as in the case of $x=12n+2$ with only exception that $s=\frac{12n+10-(2p-12)-12}{12}=\frac{6n-p+5}{6}$. For $x=12n+4$, see Lemma \ref{lemfp4} and Theorem \ref{thm14}.
\end{proof}

\noindent
Comparison tables similar to \S \ref{sec12n4} are given in Tables \ref{tabtetan2}, \ref{tabtetan8}, \ref{tabtetan10} in Appendix, where $T(n)=\frac{f(n)}{n_1}$. 

\begin{theorem}\label{thm612}
Let $n\in\Nat$. Then the degree of $x=12n+k$, $k=0,6$ in $G(x/2)$ is $n+n_1-f(n)$, where
$$f(n)=|\bigcup\limits_{\substack{p\, =\, \text{prime}\\ 3< p<\sqrt{x}}} \ \left(F(x,p,+)\cup F(x,p,-)\right)|,$$
$|F(x,p,+)\cup F(x,p,-)|=$
{\footnotesize $$\left\{\begin{array}{ll}
\lfloor (n_1-k_1)/p\rfloor+\lfloor (n-k_3)/p\rfloor+2, & \text{if } \Div{p}{x}\\
\lfloor (n_1-k_1)/p\rfloor + \lfloor (n_1-k_2)/p\rfloor +\lfloor (n-k_3)/p\rfloor + \lfloor (n-k_4)/p\rfloor +3, & \text{if } \nDiv{p}{x}\\
\end{array}\right.$$}
$k_1=[(p-[x]_p^1-b_1)\times 12^{p-2}]_p+1$, $k_2=[(x-b_1)\times 12^{p-2}]_p+1$, $k_3=[(p-[x]_p^1-b_2)\times 12^{p-2}]_p+1$, $k_4=[(x-b_2)\times 12^{p-2}]_p+1$, where $[a]_p=i\in\set{0,1,2,\ldots,p-1}$ and $[a]_p^1=i\in\set{1,2,\ldots,p}$ such that $\Mod{a}{i}{p}$, $b_1=2$, $b_2=10$ for $k=0$ and $b_1=4$, $b_2=8$ for $k=6$, $n_1=n$ for $k=0$ and $n_1=n+1$ for $k=6$ 
\end{theorem}

\begin{proof}
We first note that $3$ cannot divide both $p+1$ and $p-1$ simultaneously. Also if $\nDiv{3}{p+1}$ and $\nDiv{3}{p-1}$, then $\Div{3}{p}$ which cannot arise as $p$ is prime and $p>3$. Thus the remaining case that either $\Div{3}{p+1}$ or $\Div{3}{p-1}$ (but not both) will always occur when $\nDiv{p}{x}$.

\vspace{1em}
\noindent
Suppose $x=12n+6$, $n\in\Nat$. If $12n+5$ is prime, then $x$ can be expressed as a sum of an odd prime and $1$. But $x-3=12n+3$ is not a prime. So $x$ cannot be expressed as a sum of an odd prime and $3$. Other neighbors of $x$ in $G(x/2)$ are either of the form $12m+4$, $m\in\Int$, $0\leq m\leq n$ or of the form $12m+8$, $m\in\Int$, $0\leq m\leq n-1$. Now in order to calculate the number of neighbors of the form $12m+4$, we evaluate $k_1$ and $k_2$ as before. Here $k_1=m+1$, where $[12m+4]_p=[p-[x]_p^1]_p$ which implies $k_1=[(p-[x]_p^1-4)\times 12^{p-2}]_p+1$. Similarly, $k_2=[(x-4)\times 12^{p-2}]_p+1$ and $s=\frac{12n+6-(2p-12)-4}{12}=\frac{6n-p+7}{6}$. On the other hand, for neighbors of the form $12m+8$, we have $k_1=[(p-[x]_p^1-8)\times 12^{p-2}]_p+1$. Similarly, $k_2=[(x-8)\times 12^{p-2}]_p+1$ and $s=\frac{12n+6-(2p-12)-8}{12}=\frac{6n-p+5}{6}$.

\vspace{1em}
\noindent
Next suppose $x=12n$, $n\in\Nat$. As before, if $12n-1$ is prime, then $x$ is a sum of an odd prime and $1$. Here also $x-3=12n-3$ is not prime. So $x$ cannot be expressed as a sum of an odd prime and $3$. Other neighbors of $x$ in $G(x/2)$ are either of the form $12m+2$ or of the form $12n+10$. For $12m+2$ form, $k_1=[(p-[x]_p^1-2)\times 12^{p-2}]_p+1$. Similarly, $k_2=[(x-2)\times 12^{p-2}]_p+1$ and $s=\frac{12n-(2p-12)-2}{12}=\frac{6n-p+5}{6}$ and for $12m+10$ form, $k_1=[(p-[x]_p^1-10)\times 12^{p-2}]_p+1$. Similarly, $k_2=[(x-10)\times 12^{p-2}]_p+1$ and $s=\frac{12n-(2p-12)-10}{12}=\frac{6n-p+1}{6}$.
\end{proof}

\noindent
Comparison tables are given in Tables \ref{tabtetan0} and \ref{tabtetan6} in Appendix, where $T(n)=\frac{f(n)}{n+n_1}$. 

\vspace{1em}
\noindent
In the following we obtain another formula for the degree of an even integer $x$ in $G(x/2)$. We begin with the following definition.

\begin{definition}
Let $x$ be an even integer, $x\geq 12$. Let $PR(x)$ be the set of primes $p$ such that $3<p<\sqrt{x}$. Let $p\in PR(x)$. We define
$$g(p)=|F(x,p,+)\cup F(x,p,-)|$$
where a formula for $|F(x,p,+)\cup F(x,p,-)|$ is given in Theorems \ref{thm24810} and \ref{thm612}.\\[0.5em]
 Now for any two distinct primes $p,q\in PR(x)$, we have\\[0.5em]
$(F(x,p,+)\cup F(x,p,-))\cap (F(x,q,+)\cup F(x,q,-))$
$$= (F(x,p,+)\cap F(x,q,+))\cup (F(x,p,+)\cap F(x,q,-))\cup (F(x,p,-)\cap F(x,q,+))\cup (F(x,p,-)\cap F(x,q,-)).$$
In general, let $A\subseteq PR(x)$, $|A|>1$. We define 
$$g(A)= \left|\bigcap\limits_{p\in A} (F(x,p,+)\cup F(x,p,-))\right|= \left|\bigcup\limits_{B\subseteq A} \left\{ \left(\bigcap\limits_{p\in B} F(x,p,+)\right)\cap \left(\bigcap\limits_{p\in A\setminus B} F(x,p,-)\right)\right\}\right|.$$
We have noticed before (e.g., see Lemma \ref{lemfp4}), $F(x,p,-)\subseteq F(x,p,+)$ when $\Div{p}{x}$ and in all other cases, $F(x,p,-)\cap F(x,p,+)=\emptyset$. Let $A_1=\Set{p\in A}{\nDiv{p}{x}}$. Then we define
$$u(B)=\left(\bigcap\limits_{p\in B} F(x,p,+)\right)\cap \left(\bigcap\limits_{p\in A_1\setminus B} F(x,p,-)\right).$$
Thus we have
\begin{equation} \label{formga}
g(A)= \sum\limits_{B\subseteq A} |u(B)|=\sum\limits_{B\subseteq A} \left|\left(\bigcap\limits_{p\in B} f(x,p,+)\right)\cap \left(\bigcap\limits_{p\in A_1\setminus B} f(x,p,-)\right)\right|.
\end{equation}
\end{definition}

\vspace{1em}
\noindent
In order to calculate each term under summation, say, $u(B)$, we proceed as follows:

\vspace{1em}
\noindent
{\bf Computation of $|u(B)|$:}\  Let us first explain for $x=12n+k$, $n\in\Nat$, where $k=2,4,8$ or $10$. Let $p\in PR(x)$. Let $\Mod{x}{i}{p}$ where $1\leq i\leq p$ and $j=p-i$. Let $k_1(p)$ and $k_2(p)$ be the positions of $j$ and $i$ in the sequence $S=\Set{[12m+b]_p}{m=0,1,2,\ldots,p-1}$, where $b=6$ for $k=4,8$ and $b=12$ for $k=2,10$ (see Theorem \ref{thm24810}). Let $A\subseteq PR(x)$, $|A|>1$ and $B\subseteq A$. Let $x_p=k_1(p)$ for all $p\in B$ and $y_p=k_2(p)$ for all $p\in A_1\setminus B$. Let $h=h(B)$ be the unique solution in modulo $v=\prod\limits_{p\in A_1\cup B} p$ by Chinese Remainder Theorem such that 
$$\Mod{h}{x_p}{p} \text{ for all } p\in B \text{ and }\Mod{h}{y_p}{p} \text{ for all }p\in A_1\setminus B.$$ 
If $h=0$, then we take $h=v$. Now if $h>n$, then $|u(B)|=0$. Then to count $|u(B)|$ we have to enumerate the number of terms $y=12m+b\in S$ when $m=h-1, h+v-1, h+2v-1,\ldots, h+\lfloor\frac{n-h}{v}\rfloor v-1$ (the term $-1$ appears as we take the sequence $S$ starting from $m=0$). Also we want to exclude (from the list of non-neighbors of $x$) the last term $y\in S$ such that $\frac{x-y}{2}=p$ but $\nDiv{p}{x+y}$. Thus 
we compute $e=\frac{1}{2}(12n+k-12(h+\lfloor\frac{n-h}{v}\rfloor v-1)-b)$ $=6(n-h-\lfloor\frac{n-h}{v}\rfloor v)+\frac{k+12-b}{2}$ and $e_1=\frac{1}{2}(12n+k+12(h+\lfloor\frac{n-h}{v}\rfloor v-1)+b)$ $=6(n+h+\lfloor\frac{n-h}{v}\rfloor v)+\frac{k+12+b}{2}$ and get the following formula for $|u(B)|$:
\begin{equation}\label{formub1}
|u(B)|=\left\{
\begin{array}{ll}
\lfloor\frac{n-h}{v}\rfloor & \text{if } e \text{ is prime, } \Div{e}{v} \text{ and }\nDiv{e}{e_1}\\
\lfloor\frac{n-h}{v}\rfloor +1 & \text{otherwise}.
\end{array}\right.
\end{equation}

\noindent
In the similar way, in the case of $x=12n+k$, $n\in\Nat$, where $k=0$ or $6$, we compute $k_1(p),k_2(p)$, $k_3(p),k_4(p)$ and $b_1,b_2$ as in Theorem \ref{thm612}. Let $A\subseteq PR(x)$, $|A|>1$ and $B\subseteq A$. We compute $h=h(B)$ same as above with the values of $k_1(p)$, $k_2(p)$ and $b_1$ in place of $b$ (as in Theorem \ref{thm612}). Accordingly, we get $|u_1(B)|$ (say) as in (\ref{formub1}) (replacing $b$ by $b_1$). 

\vspace{1em}
\noindent
Next we consider $x_p=k_3(p)$ for all $p\in B$ and $y_p=k_4(p)$ for all $p\in A_1\setminus B$. Let $h_1=h_1(B)$ be the unique solution in modulo $v=\prod\limits_{p\in A_1\cup B} p$ by Chinese Remainder Theorem such that 
$$\Mod{h_1}{x_p}{p} \text{ for all } p\in B \text{ and }\Mod{h_1}{y_p}{p} \text{ for all }p\in A_1\setminus B.$$ 
If $h_1=0$, then we take $h_1=v$. Now if $h_1>n$, then $|u_2(B)|=0$.\\
 Let $e_2=6(n-h_1-\lfloor\frac{n-h_1}{v}\rfloor v)+\frac{k+12-b_2}{2}$ and $e_3=6(n+h_1+\lfloor\frac{n-h_1}{v}\rfloor v)+\frac{k+12+b_2}{2}$. We get 
\begin{equation}\label{formub2}
|u_2(B)|=\left\{
\begin{array}{ll}
\lfloor\frac{n-h_1}{v}\rfloor & \text{if } e_2 \text{ is prime, } \Div{e_2}{v} \text{ and }\nDiv{e_2}{e_3}\\
\lfloor\frac{n-h_1}{v}\rfloor +1 & \text{otherwise}.
\end{array}\right.
\end{equation}
So we have, in this case, $|u(B)|=|u_1(B)|+|u_2(B)|$. 

\vspace{1em}
\noindent
Finally, we write $g(\set{p})=g(p)$. Then we have the following formula for $f(n)$, where $f(n)$ is defined in Theorems \ref{thm24810} and \ref{thm612}.

\begin{theorem}\label{thcompfn}
Let $n\in\Nat$, $x=12n+k$, where $k=0,2,4,6,8$ or $10$. Then 
$$f(n)=\sum\limits_{\emptyset\neq A\subseteq PR(x)} (-1)^{|A|+1}\ \left(\sum\limits_{B\subseteq A} |u(B)|\right).$$
\end{theorem}

\vspace{1em}
\begin{remark}
The above theorem is analogous to the famous Legendre's formula:
\begin{equation}\label{legeqn}
\pi(x)-\pi(\sqrt{x})=\sum\limits_{\Div{d}{m(x)}} \mu(d) \left\lfloor \frac{x}{d}\right\rfloor\, -1,
\end{equation}
where $m(x)=\prod\limits_{p\leq \sqrt{x}} p$ is the product of all primes less than or equal to $\sqrt{x}$ and $\mu(d)$ is the M$\ddot{\text{o}}$bius function defined by $\mu(1)=1$, $\mu(d)=1$ if $d$ is an even number of product of distinct primes, $\mu(d)=-1$ if $d$ is an odd number of product of distinct primes and it is $0$ otherwise. Note that the formula (\ref{legeqn}) can be obtained from (\ref{eqnpix}) using a similar concept that we developed for counting $f(n)$.
\end{remark}

\section{Further Approximation}\label{secfa}

\noindent
From studies in \S \ref{sec12n4} and \S \ref{seckother}, we see that, for large $x$ we may approximate the degree of $x$ in the near Goldbach graph $G(x/2)$ in the following compact form (throughout this section $p$ stands for prime numbers only):

\begin{equation}\label{etacx1}
\eta_c(x) = \left\{%
\begin{array}{rl}
2\,\lfloor\frac{x}{12}\rfloor \prod\limits_{\Div{p}{x},\, 3<p<\sqrt{x}} \left( 1-\frac{1}{p}\right)\ \prod\limits_{\nDiv{p}{x},\, 3<p< \sqrt{x}} \left( 1-\frac{2}{p}\right) & \text{if } \Div{3}{x}\\
\lfloor\frac{x}{12}\rfloor \prod\limits_{\Div{p}{x},\, 3<p<\sqrt{x}} \left( 1-\frac{1}{p}\right)\ \prod\limits_{\nDiv{p}{x},\, 3<p< \sqrt{x}} \left( 1-\frac{2}{p}\right) & \text{if } \nDiv{3}{x}.\\
\end{array}\right.
\end{equation}

\vspace{1em}
\noindent
Now suppose\ $\Div{3}{x}$. Then $\prod\limits_{\Div{p}{x},\, 2<p<\sqrt{x}} \left( 1-\frac{1}{p}\right)$ $= (1-\frac{1}{3})\, \prod\limits_{\Div{p}{x},\, 3<p<\sqrt{x}} \left( 1-\frac{1}{p}\right)$ $=\frac{2}{3}\, \prod\limits_{\Div{p}{x},\, 3<p<\sqrt{x}} \left( 1-\frac{1}{p}\right)$. Thus $2\, \prod\limits_{\Div{p}{x},\, 3<p<\sqrt{x}} \left( 1-\frac{1}{p}\right)$ $=3\, \prod\limits_{\Div{p}{x},\, 2<p<\sqrt{x}} \left( 1-\frac{1}{p}\right)$ and $\prod\limits_{\nDiv{p}{x},\, 3<p< \sqrt{x}} \left( 1-\frac{2}{p}\right)$ $=\prod\limits_{\nDiv{p}{x},\, 2<p< \sqrt{x}} \left( 1-\frac{2}{p}\right)$ as $\Div{3}{x}$.

\vspace{1em}
\noindent
On the other hand, if $\nDiv{3}{x}$, then \\
$\prod\limits_{\nDiv{p}{x},\, 2<p< \sqrt{x}} \left( 1-\frac{2}{p}\right)$ $=(1-\frac{2}{3})\, \prod\limits_{\nDiv{p}{x},\, 3<p< \sqrt{x}} \left( 1-\frac{2}{p}\right)$ $=\frac{1}{3}\, \prod\limits_{\nDiv{p}{x},\, 3<p< \sqrt{x}} \left( 1-\frac{2}{p}\right)$. Thus $\prod\limits_{\nDiv{p}{x},\, 3<p< \sqrt{x}} \left( 1-\frac{2}{p}\right)$ $=3\,\prod\limits_{\nDiv{p}{x},\, 2<p< \sqrt{x}} \left( 1-\frac{2}{p}\right)$. Also $\prod\limits_{\Div{p}{x},\, 3<p<\sqrt{x}} \left( 1-\frac{1}{p}\right)$ $=\prod\limits_{\Div{p}{x},\, 2<p<\sqrt{x}} \left( 1-\frac{1}{p}\right)$ as $\nDiv{3}{x}$.

\vspace{1em}
\noindent
Therefore we can write the formula in the following single equation
\begin{equation}\label{etacx22}
\eta_c(x) = 3\, \lfloor\frac{x}{12}\rfloor \prod\limits_{\Div{p}{x},\, 2<p<\sqrt{x}} \left( 1-\frac{1}{p}\right)\ \prod\limits_{\nDiv{p}{x},\, 2<p< \sqrt{x}} \left( 1-\frac{2}{p}\right).
\end{equation}

\vspace{1em}
\noindent
Again \\ 
{\footnotesize $\prod\limits_{\Div{p}{x},\, 2<p<\sqrt{x}} \left(1-\frac{1}{p}\right)\ \prod\limits_{\nDiv{p}{x},\, 2<p<\sqrt{x}} \left(1-\frac{2}{p}\right)=\prod\limits_{\Div{p}{x},\, 2<p<\sqrt{x}} \frac{1-\frac{1}{p}}{1-\frac{2}{p}}\ \prod\limits_{2<p<\sqrt{x}} \left(1-\frac{2}{p}\right)=\prod\limits_{\Div{p}{x},\, 2<p<\sqrt{x}} \frac{p-1}{p-2}\ \prod\limits_{2<p<\sqrt{x}} \left(1-\frac{2}{p}\right)$.}

\vspace{1em}
\noindent
Thus we have 

\begin{equation}\label{etacx23}
\eta_c(x) = 3\, \lfloor\frac{x}{12}\rfloor \prod\limits_{\Div{p}{x},\, 2<p<\sqrt{x}} \frac{p-1}{p-2}\ \prod\limits_{2<p< \sqrt{x}} \left( 1-\frac{2}{p}\right).
\end{equation}

\noindent
In the following we wish to find a more compact closed form of $\eta_c(x)$. We begin with stating the famous Mertens' theorem:

\begin{theorem}\label{thmmert} {\em \cite{HWR, MERT}}
$$\prod\limits_{p\leq x} \left( 1-\frac{1}{p}\right)\sim \frac{e^{-\gamma}}{\log\, x},$$
where $\gamma = 0.5772156649$ is known as Euler's constant.
\end{theorem} 

\noindent
Now we wish to prove the following:

\begin{theorem}\label{2bypform}
$$\prod\limits_{2<p\leq x} \left( 1-\frac{2}{p}\right)\sim \frac{e^{-\beta}}{(\log\, x)^2}, \text{ where }\beta =0.183407.$$
\end{theorem}

\begin{proof}
Let $z=\prod\limits_{2<p\leq x} \left( 1-\frac{2}{p}\right)$ and $L=\log z$. Then

$${L=\sum\limits_{2<p\leq x} \log\, \left( 1-\frac{2}{p}\right)}.$$

\noindent
We know that
$$\log\, (1-u) = -\left( u + \frac{u^2}{2} + \frac{u^3}{3} + \frac{u^4}{4} +\cdots\right)\ \text{for } |u|<1.$$
So we take $u=\frac{2}{p}$ and examine the terms $\displaystyle{L_i=\sum\limits_{2<p\leq x} \frac{(2/p)^i}{i}}$ for $i\in\Nat$

\vspace{1em}
\noindent
Now
$$L_1=\sum\limits_{2<p\leq x} \frac{2}{p} = 2 \sum\limits_{2<p\leq x} \frac{1}{p}.$$

\noindent
We know that 
$$\sum\limits_{p\leq x} \frac{1}{p} = \log\log\, x+B_1+o(1)$$
where $B_1=\gamma +\sum\limits_{p\leq x} \left\{\log\left(1-\frac{1}{p}\right) + \frac{1}{p}\right\}$. The series $B_1$ converges and equals to $0.2614972128$ (approximately) \cite{HWR}. Thus we have 
$$L_1= 2 \left(\log\log\, x + B_1 - \frac{1}{2}\right) = 2\log\log\, x+2B_1-1.$$

\noindent
Next we consider for $i>1$
$$L_i=\sum\limits_{2<p\leq x} \frac{(2/p)^i}{i}.$$
We have $$\sum\limits_{i=2}^{\infty} L_i = \sum\limits_{i=2}^{\infty} \sum\limits_{2<p\leq x} \frac{(2/p)^i}{i} = \sum\limits_{2<p\leq x} \sum\limits_{i=2}^{\infty} \frac{(2/p)^i}{i}.$$

\noindent
Now the series $\sum\limits_{i=2}^{\infty} \frac{(2/p)^i}{i}$ is convergent for every prime $p>2$. Since the other sum is finite the whole series $\sum\limits_{i=2}^{\infty} L_i$ is convergent. In the following let us compute the sum. We know that the {\em prime zeta function} 
$$P(s)= \sum\limits_{p\,\in\, primes} \frac{1}{p^s}$$
is convergent for all $s\geq 2$. Now
$$\sum\limits_{i=2}^{\infty} \sum\limits_{2<p\leq x} \frac{(2/p)^i}{i} = \sum\limits_{i=2}^{\infty} \frac{2^i}{i}\left\{\sum\limits_{p>2} \frac{1}{p^i}-\sum\limits_{p>x} \frac{1}{p^i}\right\} = \sum\limits_{i=2}^{\infty} \frac{2^i}{i}\left\{\left(P(i)-\frac{1}{2^i}\right)-\sum\limits_{p>x} \frac{1}{p^i}\right\}.$$
We note that $\sum\limits_{p>x} \frac{1}{p^i}\rightarrow 0$ as $x\rightarrow \infty$. Also for large $x$ this term becomes insignificant. We define 
$$R:=\sum\limits_{i=2}^{\infty} \frac{2^i}{i}\left(P(i)-\frac{1}{2^i}\right).$$
Then $$L=\log\, z=- 2\log\log\, x-2B_1-R+1+o(1).$$
This implies 
$$z\sim \frac{e^{1-2B_1-R}}{(\log\, x)^2}.$$
Finally, to compute $R$, we note that $L_i$ becomes insignificant for large $i$ relative to the sum of its first few small values. In fact, the function $\displaystyle{\sum\limits_{i=2}^n \frac{2^i}{i} \left(P(i)-\frac{1}{2^i}\right)}$ approximately equals to $0.660413$ for $30\leq n\leq 50$. Thus we may take $R\approx 0.660413$. Then $1-2B_1-R=1-2\times 0.261497-0.660413=-0.183407$ as required.
\end{proof}

\noindent 
Table \ref{tasymp} clearly shows the asymptotic nature proved in the above theorem.

\begin{table}[ht]
{\footnotesize $$\begin{array}{|c|l|l|l|l|l|}
\hline
x & 10 & 10^2 & 10^3 & 10^4 & 10^5 \\ 
\hline
z & 0.142857 & 0.038297 & 0.0173126 & 0.00978883 & 0.00627642\\ 
\hline
\frac{e^{-\beta}}{(\log\, x)^2} & 0.157006 & 0.0392515 & 0.0174451 & 0.00981287 & 0.00628024 \\
\hline
\frac{e^{-\beta}/(\log\, x)^2}{z} & 1.0990431 & 1.02492362 & 1.0076534 & 1.0024559 & 1.0006609 \\
\hline
\hline
x & 10^6 & 10^7 & 10^8 & 10^9 & 10^{10}\\
\hline
z & 0.00436093 & 0.00320414 & 0.0024532 &  0.001938338 & 0.0015700567 \\
\hline
\frac{e^{-\beta}}{(\log\, x)^2} &  0.00436128 & 0.0032042 & 0.00245322 & 0.001938344 & 0.00157006\\
\hline
\frac{e^{-\beta}/(\log\, x)^2}{z} & 
 1.0000802 & 1.0000187 & 1.0000082 & 1.0000031 & 1.0000021\\
\hline
\end{array}$$}
\caption{Comparison between $z$ and $\frac{e^{-\beta}}{(\log\, x)^2}$.}\label{tasymp}
\end{table}

\vspace{1em}
\begin{theorem}\label{thmetaa}
For large positive even integer $x$, the degree of $x$ in the near Goldbach graph $G(x/2)$ is approximately given by
$$\eta(x)=\kappa(x) \frac{x\, e^{-\beta}}{(\log\, x)^2}.$$
where $\beta=0.183407$, $\displaystyle{\kappa(x)=\prod\limits_{\Div{p}{x},\, p>2} \frac{p-1}{p-2}}$.
\end{theorem}

\begin{proof}
We have, by (\ref{etacx23}), the required degree is approximately
$$3\, \lfloor\frac{x}{12}\rfloor \prod\limits_{\Div{p}{x},\, 2<p<\sqrt{x}} \frac{p-1}{p-2}\ \prod\limits_{2<p< \sqrt{x}} \left( 1-\frac{2}{p}\right).$$
We first note that, 
$$\prod\limits_{\Div{p}{x},\, p>2} \frac{p-1}{p-2} = \prod\limits_{\Div{p}{x},\, 2<p<x} \frac{p-1}{p-2} = \prod\limits_{\Div{p}{x},\, 2<p<\sqrt{x}} \frac{p-1}{p-2}\, \prod\limits_{\Div{p}{x},\, \sqrt{x}<p<x} \frac{p-1}{p-2}.$$

\noindent
It is interesting to note that if there is a prime $p$ such that $\Div{p}{x}$ and $p>\sqrt{x}$, then $\frac{x}{p}<\frac{x}{\sqrt{x}}=\sqrt{x}$. Thus for any $x$, there can be at most one prime factor $p$ of $x$ such that $p>\sqrt{x}$. So for large $x$, $\prod\limits_{\Div{p}{x},\, \sqrt{x}<p<x} \frac{p-1}{p-2}\sim 1$. Thus $\prod\limits_{\Div{p}{x},\, 2<p<\sqrt{x}} \frac{p-1}{p-2}\sim \prod\limits_{\Div{p}{x},\, p>2} \frac{p-1}{p-2}=\kappa(x)$.

\vspace{1em}
\noindent
Now by Theorem \ref{2bypform}, we have $\prod\limits_{2<p\leq x} \left( 1-\frac{2}{p}\right)\sim \frac{e^{-\beta}}{(\log\, x)^2}$. Then $\prod\limits_{2<p< \sqrt{x}} \left( 1-\frac{2}{p}\right)$ $=\prod\limits_{2<p\leq \sqrt{x}} \left( 1-\frac{2}{p}\right)$ (as $p$ is odd prime and $x$ is even) $=\frac{e^{-\beta}}{(\log\, \sqrt{x})^2}$ $=\frac{4e^{-\beta}}{(\log\, x)^2}$. Thus the required degree is (approximately) $\kappa(x) \frac{x\, e^{-\beta}}{(\log\, x)^2}$ as required.
\end{proof}

\noindent
Table \ref{tdegcx} shows the comparison between exact values and approximated values of degrees of vertices (even integers) $x$ in the near Goldbach graph $G(x/2)$:

\begin{table}[ht]
$$\begin{array}{|c|l|l|l|l|l|l|l|l|l|}
\hline
x & 10^2 & 10^3 & 10^4 & 10^5 & 10^6 & 10^7 & 10^8 & 10^9 & 10^{10} \\
\hline
\text{deg}(x) & 6 & 28 & 127 & 810 & 5402 & 38807 & 291400 & 2274205 & 18200488 \\
\hline
\eta_c(x) & 4 & 20 & 127 & 820 & 5770 & 42642 & 326294 & 2582599 & 20921398 \\
\hline
\eta(x) & 5 & 23 & 130 & 837 & 5815 & 42722 & 327095 & 2584459 & 20934120\\
\hline
\end{array}$$
\caption{Comparison between exact values and approximated values of degree of $x$ in $G(x/2)$.}\label{tdegcx}
\end{table}

\noindent
It is known that if $D(x)$ is the number of ways that $x$ is represented as the sum of two odd primes (where $p+q$ and $q+p$ are considered as different representations), then the following is a conjecture by Hardy and Littlewood \cite{HDLW, Wu}:
$$D(x)\sim 2 \frac{C(x) x}{(\log\, x)^2},\ \text{where } C(x)=\prod\limits_{\Div{p}{x},\, p>2} \frac{p-1}{p-2}\ \prod\limits_{p>2} \left( 1-\frac{1}{(p-1)^2}\right).$$

\noindent
Now $\prod\limits_{p>2} \left( 1-\frac{1}{(p-1)^2}\right)$ is convergent and is approximately equal to $0.660162$. Thus 
$$C(x)=0.660162\, \prod\limits_{\Div{p}{x},\, p>2} \frac{p-1}{p-2}=0.660162\ \kappa(x).$$
Thus in terms of $D(x)$ we observe that
$$\eta(x)=\frac{e^{-\beta}}{0.660162}\,\frac{D(x)}{2} = 1.26095\,\frac{D(x)}{2}.$$

\noindent
In the following, we prove this result formally.

\begin{corollary}
$$\eta(x)= 4e^{-2\gamma}\,\frac{D(x)}{2}=1.26095\,\frac{D(x)}{2}=O(D(x)).$$
\end{corollary}

\begin{proof}

\noindent
First we notice that
$$1-\frac{1}{(p-1)^2}\ =\ \frac{p(p-2)}{(p-1)^2}\ =\ \frac{\left(1-\frac{2}{p}\right)}{\left(1-\frac{1}{p}\right)^2}.$$

\noindent
Thus
$$\prod\limits_{2<p<\sqrt{x}} \left(1-\frac{1}{(p-1)^2}\right)=\frac{\prod\limits_{2<p<\sqrt{x}} \left(1-\frac{2}{p}\right)}{\prod\limits_{2<p<\sqrt{x}}\left(1-\frac{1}{p}\right)^2}=\frac{\prod\limits_{2<p<\sqrt{x}} \left(1-\frac{2}{p}\right)}{4\,\prod\limits_{p<\sqrt{x}}\left(1-\frac{1}{p}\right)^2}\sim \frac{\frac{e^{-\beta}}{(\log\,\sqrt{x})^2}}{4\,\frac{e^{-2\gamma}}{(\log\,\sqrt{x})^2}}= \frac{e^{2\gamma-\beta}}{4}$$
for large value of $x$, by Theorems \ref{thmmert} and \ref{2bypform}.

\noindent
Now 
$$\prod\limits_{p>2} \left(1-\frac{1}{(p-1)^2}\right) = \prod\limits_{2<p<\sqrt{x}} \left(1-\frac{1}{(p-1)^2}\right) \prod\limits_{p>\sqrt{x}} \left(1-\frac{1}{(p-1)^2}\right)$$

\noindent
and $\prod\limits_{p>\sqrt{x}} \left(1-\frac{1}{(p-1)^2}\right)\rightarrow 1$ as $x\rightarrow\infty$. Thus we have 
$$\prod\limits_{p>2} \left(1-\frac{1}{(p-1)^2}\right) = \frac{e^{2\gamma-\beta}}{4}=0.660162.$$

\noindent
Thus we have $C(x)=\frac{e^{2\gamma-\beta}}{4}\,\kappa(x)$. Therefore
$$\eta(x)=\kappa(x) \frac{x\, e^{-\beta}}{(\log\, x)^2}=4\, e^{-2\gamma}\, C(x)\,\frac{x}{(\log\, x)^2}=4\, e^{-2\gamma}\,\frac{D(x)}{2} =1.26095\, \frac{D(x)}{2}=O(D(x)).$$

\end{proof}

\begin{remark}\label{rem:HL}
The gap $4e^{-2\gamma}$ is not surprising and quite natural as in Formula (\ref{etacx1}) of $\eta_c(x)$ we assumed the events $A_p$ (see Definition \ref{defap}) are mutually independent for different primes $p$. But it is not so in reality. In fact, if we follow the distribution of primes as in Prime Number Theorem, then we will find the same factor. This is also mentioned by Hardy in his book \cite{HWR} while giving a justification of his conjecture on twin primes. Here we provide a similar justification which is worth noting. 

\vspace{1em}
\noindent
Let $x$ be an integer. Let $S=\set{1,2,3,\ldots,x}$. We wish to count $\pi(x)$, the number of primes less than or equal to $x$. Let $p$ be a prime number such that $2\leq p\leq\sqrt{x}$. Let $F(x,p)=\Set{y\in S}{p\mid y,\, y>p}$. Then $|F(x,p)|=\lfloor\frac{x-p}{p}\rfloor$. Therefore 
\begin{equation}\label{eqnpix}
\pi(x)=x-\left|\,\bigcup\limits_{p\leq \sqrt{x}} F(x,p)\,\right|-1.
\end{equation}
Let $A_p$ be the event that $y\in S$ satisfies $p\mid y$ and $y>p$ for a prime number $p\leq\sqrt{x}$. Then the probability that $A_p$ occurs is given by $P(A_p)=\frac{1}{x}\,\lfloor\frac{x-p}{p}\rfloor\approx\frac{1}{p}$ for large $x$. Let $$A=\bigcup\limits_{p\leq\sqrt{x}} A_p.$$
Now if we assume the events $A_p$ are mutually independent, then the probability of $A$ is given by
$$1-\prod\limits_{p\leq\sqrt{x}} \left( 1-\frac{1}{p}\right).$$
Suppose $M(x)$ be the approximated value of the number of primes less than or equal to $x$ considering the events $A_p$ are mutually independent. Then 
$$M(x) = x\,\prod\limits_{p\leq\sqrt{x}} \left( 1-\frac{1}{p}\right) \sim x\cdot\frac{e^{-\gamma}}{\log \sqrt{x}} \text{ (by Mertens' Theorem)}$$ 
$$= 2 e^{-\gamma}\cdot\frac{x}{\log x} \sim 2 e^{-\gamma}\pi(x) \text{ (by Prime Number Theorem)}.$$
Thus we have 
\begin{equation}\label{dcf}
\pi(x)\sim \frac{1}{2} e^\gamma\ M(x).
\end{equation}
We call the number $\frac{1}{2} e^\gamma$ as the {\em discrete correction factor}. Now by (\ref{etacx22}) we have the degree of an even integer $x$ in the near Goldbach graph $G(x/2)$, considering that divisibility events are mutually independent for different primes, is given by

\vspace{1em}
\noindent
$\displaystyle{\eta_c(x) = 3\, \lfloor\frac{x}{12}\rfloor \prod\limits_{\Div{p}{x},\, 2<p<\sqrt{x}} \left( 1-\frac{1}{p}\right)\ \prod\limits_{\nDiv{p}{x},\, 2<p< \sqrt{x}} \left( 1-\frac{2}{p}\right).}$\\[1em]
$\displaystyle{= 3\, \lfloor\frac{x}{12}\rfloor \ \frac{\prod\limits_{\Div{p}{x},\, 2<p<\sqrt{x}} \left( 1-\frac{1}{p}\right)}{\prod\limits_{\Div{p}{x},\, 2<p< \sqrt{x}} \left( 1-\frac{2}{p}\right)} \ \prod\limits_{2<p< \sqrt{x}} \left( 1-\frac{2}{p}\right).}$\\[1em]
$\displaystyle{ = 3\, \lfloor\frac{x}{12}\rfloor \ \prod\limits_{\Div{p}{x},\, 2<p<\sqrt{x}} \frac{p-1}{p-2}\  \frac{\prod\limits_{2<p< \sqrt{x}} \left( 1-\frac{2}{p}\right)}{\prod\limits_{2<p< \sqrt{x}} \left( 1-\frac{1}{p}\right)^2} \ \prod\limits_{2<p< \sqrt{x}} \left( 1-\frac{1}{p}\right)^2.}$\\[1em]
$\displaystyle{ = 3\, \lfloor\frac{x}{12}\rfloor \ \prod\limits_{\Div{p}{x},\, 2<p<\sqrt{x}} \frac{p-1}{p-2} \ \prod\limits_{2<p< \sqrt{x}} \frac{p(p-2)}{(p-1)^2}\  \prod\limits_{2<p< \sqrt{x}} \left( 1-\frac{1}{p}\right)^2.}$\\[1em]
$\displaystyle{ = 3\, \lfloor\frac{x}{12}\rfloor \ \prod\limits_{\Div{p}{x},\, 2<p<\sqrt{x}} \frac{p-1}{p-2} \ \prod\limits_{2<p< \sqrt{x}} \left( 1-\frac{1}{(p-1)^2}\right) \ \prod\limits_{2<p< \sqrt{x}} \left( 1-\frac{1}{p}\right)^2.}$\\[1em]
$\displaystyle{ = 3\, \lfloor\frac{x}{12}\rfloor \ \prod\limits_{\Div{p}{x},\, 2<p<\sqrt{x}} \frac{p-1}{p-2} \ \prod\limits_{2<p< \sqrt{x}} \left( 1-\frac{1}{(p-1)^2}\right) \cdot 4\prod\limits_{p< \sqrt{x}} \left( 1-\frac{1}{p}\right)^2.}$\\[1em]
$\displaystyle{ = 12\, \lfloor\frac{x}{12}\rfloor \ \prod\limits_{\Div{p}{x},\, 2<p<\sqrt{x}} \frac{p-1}{p-2}\  \prod\limits_{2<p< \sqrt{x}} \left( 1-\frac{1}{(p-1)^2}\right) \ \prod\limits_{p< \sqrt{x}} \left( 1-\frac{1}{p}\right)^2.}$

\vspace{1em}
\noindent
Now $\prod\limits_{\Div{p}{x},\, 2<p<\sqrt{x}} \frac{p-1}{p-2} = \prod\limits_{\Div{p}{x},\, p>2} \frac{p-1}{p-2} \prod\limits_{\Div{p}{x},\, p>\sqrt{x}} \frac{p-2}{p-1}\sim \prod\limits_{\Div{p}{x},\, p>2} \frac{p-1}{p-2}=\kappa(x)$\\[1em] 
since $\prod\limits_{\Div{p}{x},\, p>\sqrt{x}} \frac{p-2}{p-1}\rightarrow 1$ as $x\rightarrow\infty$. Note that there can be at most one such $p\mid x$ and $p>\sqrt{x}$ as then $\frac{x}{p}<\sqrt{x}$.

\vspace{1em}
\noindent
Next since $\prod\limits_{p>2} \left( 1-\frac{1}{(p-1)^2}\right)$ is convergent and equal to $0.660162=\Pi_2$ (say),\\[1em]
we have $\prod\limits_{2<p< \sqrt{x}} \left( 1-\frac{1}{(p-1)^2}\right)\sim \Pi_2$ for large $x$. 

\vspace{1em}
\noindent
Finally, following (\ref{dcf}), we multiply the term $\prod\limits_{p< \sqrt{x}} \left( 1-\frac{1}{p}\right)^2$ with the discrete correction factor and replace it by $\left(\frac{e^{-\gamma}}{\log \sqrt{x}}\right)^2 \left(\frac{1}{2} e^{\gamma}\right)^2$. 

\vspace{1em}
\noindent
Thus, the corrected value of $d(x)$, the degree of $x$ in the near Goldbach graph $G(x/2)$, is given by 
$$d(x)\sim x\, \kappa(x)\,\Pi_2\ \left(\frac{e^{-\gamma}}{\log \sqrt{x}}\right)^2 \left(\frac{1}{2} e^{\gamma}\right)^2=\frac{x}{(\log x)^2} \kappa(x)\, \Pi_2.$$
Now it follows from the definition of near Goldbach graphs that 
$$d(x)=\left\{%
\begin{array}{ll}
\frac{1}{2} D(x)+\frac{1}{2} & x-1 \text{ and }\frac{x}{2}\text{ are primes}.\\
\frac{1}{2} D(x)+1 & x-1 \text{ is prime and }\frac{x}{2}\text{ is not prime}.\\
\frac{1}{2} D(x)-\frac{1}{2} & x-1 \text{ is not prime and }\frac{x}{2}\text{ is prime}.\\
\frac{1}{2} D(x) & x-1 \text{ and }\frac{x}{2}\text{ are not primes}.\\
\end{array}\right.$$
Thus for an even integer $x$,
$$|d(x)- \frac{1}{2} D(x)|\leq 1.$$
Therefore we have
$$D(x)\sim 2\,\kappa(x)\,\Pi_2\ \frac{x}{(\log x)^2}.$$
Though this (heuristic) arguments justify the Hardy-Littlewood conjecture, this does not provide a proof for it as the discrete correction factor and its application on powers of $\left(1-\frac{1}{p}\right)$ (as in the last step) are not yet proved rigorously.
\end{remark}

\section{Sufficient conditions}

\noindent
In this section, our target is to show analytically $T(n)<1$ for all $n$. For this we are trying to find a function that is greater than $T(n)$ but less than $1$. In the following we wish to have a sufficient condition so that we can construct such a function.  

\vspace{1em}
\noindent
Let $n\in\Nat$, $n>3$, $x=12n+4$, $S(n)=\Set{12m+6}{m\in\Nat\cup\set{0},\ m<n}$, $r_x=|PR(x)|$, where $PR(x)=\Set{p}{p \text{ is prime},\ 3<p<\sqrt{x}}$. For each prime $p\in PR(x)$, let $A_{p}$ be the event which is the set of numbers $y$ in $S(n)$ such that either $\Div{p}{x+y}$ or $\Div{p}{x-y}$ (except the case where $p=\frac{x-y}{2}$). Let $\mathscr{A}(x)=\Set{A_p}{p\in PR(x)}$ be the set of divisibility events for $x$ (see Definition \ref{defap}). The following example shows that divisibility events for two different sets of primes do not behave similarly.

\begin{example}
For $n=100$ and $x=12n+4$, the number of cases of $y=12m+6$, $m\in\Nat\cup\set{0}$, $m<n$ such that $7$ divides $x+y$ or $x-y$ (except $\frac{x-y}{2}=7$) is actually $14$ and $13$ divides $x+y$ or $x-y$ (except $\frac{x-y}{2}=13$) in $15$ cases. The number of $x+y$ or $x-y$ those are divisible by both $7$ and $13$ is $2$. Thus 
$$P(A_{7}\cap A_{13})=\frac{2}{100}=0.02 \text{ and } P(A_{7})P(A_{13})=\frac{14}{100}\times\frac{15}{100}=0.14\times 0.15=0.021.$$
So we have $P(A_{7}\cap A_{13})<P(A_{7})P(A_{13})$.

\vspace{1em}
\noindent
Next we see that the number of same cases for $11$ is $18$ and there are $3$ cases where $x+y$ or $x-y$ are divisible by both $7$ and $11$. Thus, in this case, 
$$P(A_7\cap A_{11})=0.03>0.0252=0.14\times 0.18=P(A_7)P(A_{11}).$$
It is interesting to note that for $n=315$, 
$$P(A_7)=\frac{90}{315},\ P(A_{43})=\frac{7}{315}\text{ and }P(A_7\cap A_{43})=\frac{2}{315}=P(A_7)P(A_{43}).$$
\end{example}

\noindent
In fact, it is more difficult to estimate $P(A\cap B)$ in terms of $P(A)$ and $P(B)$ when $A$ is the union of some divisibility events and $B$ is a particular divisibility event. As we see in the above example that it varies on particular cases. But the variation is not far away from what it should be if they were independent. This idea motivates us to define the following: 

\begin{definition}
Let $\mathscr{A}=\set{A_1,A_2,\ldots,A_r}$ be a collection of $r$ events. Let $\lambda (A,B)=\frac{P(B)-P(A\cap B)}{P(B)-P(A)P(B)}$, $\lambda_i=\lambda (\bigcup\limits_{j=1}^{i-1} A_j,A_i)$ and $\lambda_0=\max\,\Set{\lambda_i}{i=2,3,\ldots,r}$. Then $\mathscr{A}$ is called a set of {\em nearly independent} events if 
$$\lambda_0<2.5.$$
Note that, for independent events $A,B$, $\lambda (A,B)=1$. 
\end{definition}

\noindent
We see that Table \ref{tablam0} shows that $\mathscr{A}(x)$ is a set of nearly independent events for $x=12n+4$, $n=10^i$, $1\leq i\leq 10$ and for all $x$ up to $2\times 10^7$. The maximum value of $\lambda_0$ is found to be $2.34968$ for $x=6905498$ in this range. Also we note that cases for which $\lambda_0>2$ are very rare.

\begin{table}[ht]
$$\begin{array}{|c|l|l|l|l|l|l|l|l|l|l|}
\hline
n & 10 & 100 & 1000 & 10000 & 10^5 & 10^6 & 10^7 & 10^8 & 10^9 & 10^{10}\\
\hline
\lambda_0 & 1 & 1.17647 & 1.45289 & 1.52761 & 1.47894 & 1.79971 & 1.58345 & 1.53435 & 1.55956 & 1.56022\\
\hline
\end{array}$$
\caption{Values of $\lambda_0$ for $x=12n+4$ for $n=10$ to $n=10^{10}$.}\label{tablam0}
\end{table}

\begin{table}[ht]
$$\begin{array}{|c|c|c|c|c|c|}
\hline
\text{Range of }x & 10 - 10^3 & 10^3 - 10^4 & 10^4 - 10^5 & 10^5 - 10^6 & 10^6 - 2\times 10^7\\
\hline
\max\,(\lambda_0) & 2.14286 & 2.2222 & 2.24921 & 2.22907 & 2.34968\\
\hline
\text{At }x=& 182 & 4322 & 15692 & 542816 & 6905498\\
\hline
\end{array}$$
\caption{Maximum values of $\lambda_0$ for various ranges of $x$.}\label{tablamx}
\end{table}

\vspace{1em}
\noindent
In general, we prove the following lemmas that lead to the proof of Theorem \ref{tgood}.

\begin{lemma}\label{lemins}
Let $A$ and $B$ be two independent events that occur with probabilities $P(A)$ and $P(B)$ respectively. Let $p,q\in [0,1]$ such that $P(A)<p$ and $P(B)\leq q$. Then $P(A\cup B)<1 - (1 - p) (1 -q)$.
\end{lemma}

\begin{proof}
We have $P(A)<p$ and $P(B)\leq q$. Now 
$$P(A\cup B)=P(A)+P(B)-P(A\cap B)=P(A)+P(B)-P(A)P(B).$$
Again since $P(A)<p$ and $P(B)\leq q$, we have 
$$(1-P(A))(1-P(B))>(1-p)(1-q)\Longrightarrow P(A)+P(B)-P(A)P(B)<1-(1-p)(1-q).$$
Thus $P(A\cup B)<1-(1-p)(1-q)$ as required.
\end{proof}

\begin{lemma}\label{lemk}
Let $A$ and $B$ be two events that occur with probabilities $P(A)$ and $P(B)$ respectively\\[0.5em]
 such that $0<P(A),P(B)<1$ and $\frac{P(B)-P(A\cap B)}{P(B)-P(A)P(B)}<k$. Then $P(A\cup B)<1-(1-P(A))(1-kP(B))$.
\end{lemma}

\begin{proof}
$$\frac{P(B)-P(A\cap B)}{P(B)-P(A)P(B)}<k$$
$$\Longrightarrow P(B)-P(A\cap B)<kP(B)-kP(A)P(B))\ \text{as } P(B)-P(A)P(B)=P(B)(1-P(A))>0$$
$$\Longrightarrow P(A\cup B)=P(A)+P(B)-P(A\cap B)<P(A)+kP(B)-P(A)(kP(B))=1-(1-P(A))(1-kP(B)).$$
\end{proof}

\begin{definition}
Let $n\geq 4$ and $x=12n+4$. We define
$$\tau(n)=1-(1-pr_1(5))\prod\limits_{\substack{p\, =\, \text{prime}\\ 5<p<\sqrt{x}}} (1-2.5\, pr_1(p)),$$
where
$$pr_1(p)=\left\{\begin{array}{ll}
\frac{\lfloor (n-k_1)/p\rfloor+1}{n}, & \text{if } \Div{p}{x}\\
\frac{\lfloor (n-k_1)/p\rfloor + \lfloor (n-k_2)/p\rfloor +1}{n}, & \text{if } \nDiv{p}{x}\text{ and } \Div{3}{p+1}\\
\frac{\lfloor (n-k_1)/p\rfloor + \lfloor (n-k_2)/p\rfloor +2}{n}, & \text{if } \nDiv{p}{x}\text{ and } \nDiv{3}{p+1},\\
\end{array}\right.$$
$k_1=[(p-[x]_p^1-6)\times 12^{p-2}]_p+1$ and $k_2=[(x-6)\times 12^{p-2}]_p+1$, where $[a]_p=i\in\set{0,1,2,\ldots,p-1}$ and $[a]_p^1=i\in\set{1,2,\ldots,p}$ such that $\Mod{a}{i}{p}$ (see Definition \ref{defeta13} and Lemma \ref{lemfp4}).
\end{definition}

\begin{lemma}\label{lem1good}
Let $n\in\Nat$, $n>17$, then $0<T(n)<\tau(n)<1$.
\end{lemma}

\begin{proof}
We first note that, since $1\leq k_1,k_2\leq p_i$, $pr_1(p_i)\leq \frac{2}{p_i}+\frac{2}{n}$ for all $p_i\in PR(x)$. Since $p_i\geq 7$ for all $i\geq 2$, we have $2.5\, pr_1(p_i)\leq \frac{5}{2}\, (\frac{2}{p_i}+\frac{2}{n})=\frac{5}{p_i}+\frac{5}{n}<1$ for $n>17$. So $(1-2.5\, pr_1(p_i))\in (0,1)$. Also $pr_1(5)\in (0,1)$. So $\tau(n)\in (0,1)$. 

\vspace{1em}
\noindent
Now $T(n)=\frac{f(n)}{n}$, where $f(n)=|\bigcup\limits_{\substack{p\, =\, \text{prime}\\ 3<p<\sqrt{x}}} \ \left(F(x,p,+)\cup F(x,p,-)\right)|$ by (\ref{eqfn}).

\vspace{1em}
\noindent
 Following Definition \ref{defap}, for any prime $p_i$, ($3<p_i<\sqrt{x}$),\\
 $P(A_{p_i})=|F(x,p_i,+)\cup F(x,p_i,-)|/n=pr_1(p_i)$ (see Lemma \ref{lemfp4} and Definition \ref{defeta13}) and 
$$T(n)=P(\bigcup\limits_{i=1}^{r_x} A_{p_i}).$$
We prove $T(n)<\tau(n)$ by induction. 

\vspace{1em}
\noindent
Let
$$P_{t}=1-(1-pr_1(5))\prod\limits_{i=2}^{t} (1-2.5\, pr_1(p_i)).$$
We have $P(A_5)=pr_1(5)=P_1$. Let $A=A_5$ and $B=A_7$. Now $\lambda_2=\lambda(A,B)\leq \lambda_0<2.5$. Then $P(A\cup B)<1-(1-P(A))(1-2.5\, P(B))$ by Lemma \ref{lemk} considering $k=2.5$.\\
 Thus $P(A_5\cup A_7)<1-(1-pr_1(5))(1-2.5\, pr_1(7))=P_2$. In fact, we can prove this directly as follows:
$$P(A_5\cup A_7)=P(A_5)+P(A_7)-P(A_5\cap A_7)\leq P(A_5)+P(A_7)=pr_1(5)+pr_1(7).$$

\noindent
Now $1-(1-pr_1(5))(1-2.5\, pr_1(7))=pr_1(5)+2.5\, pr_1(7)-2.5\, pr_1(5)\, pr_1(7)$\\[0.5em]
\null\hspace{2.3in} $= \left(pr_1(5)+pr_1(7)\right)+\left(1.5-2.5\, pr_1(5)\right)\, pr_1(7)$.\\[0.5em]
Also $pr_1(5)\leq \frac{2}{5}+\frac{2}{n}<0.6=\frac{1.5}{2.5}$ (for $n> 10$) $\Longrightarrow \left(1.5-2.5\, pr_1(5)\right)>0$\\[0.5em]
$\Longrightarrow 1-(1-pr_1(5))(1-2.5\, pr_1(7))>pr_1(5)+pr_1(7)\geq P(A_5\cup A_7) \text{ as }pr_1(7)>0.$
Thus we proved that 
$$P(\bigcup\limits_{i=1}^{2} A_{p_i}) <P_2.$$

\vspace{1em}
\noindent
Let $A=\bigcup\limits_{i=1}^{t-1} A_{p_i}$ and $B=P(A_{p_t})=pr_1(p_t)$ for some $t>2$. Suppose $P(A)<P_{t-1}$. Now $\lambda_t=\lambda(A,B)\leq \lambda_0<2.5$. Then 
$$P(A\cup B)<1-(1-P(A))(1-2.5\, P(B))$$
by Lemma \ref{lemk} considering $k=2.5$. Therefore 
$$P(\bigcup\limits_{i=1}^{t} A_{p_i})=P(A\cap B)<P(A)+2.5\, P(B)-2.5\, P(A)P(B)<1-(1-P_{t-1})(1-2.5\, pr_1(p_t))=P_t$$
by Lemma \ref{lemins} as $P(A)<P_{t-1}$ and $P(B)=pr_1(p_t)$. This completes the induction. Thus finally, we have 
$$T(n)=P(\bigcup\limits_{i=1}^{r_x} A_{p_i})<\tau(n).$$
Therefore, $0<T(n)<\tau(n)<1$ as required.
\end{proof}

\noindent
Table \ref{tabtnetafn} shows the comparison between $T(n)$ and $\tau(n)$ for $x=12n+4$.

\begin{lemma}\label{lem2good}
Let $n\in\Nat$, then $\tau(n)<1-\frac{1}{n}$ for all $n>18964$ and $\tau(n)<1-\frac{2}{n}$ for all $n>71241$. 
\end{lemma}

\begin{proof}
Let
$$z(n)=(1-pr_1(5))\prod\limits_{\substack{p\, =\, \text{prime}\\ 5<p<\sqrt{12n+4}}} (1-2.5\, pr_1(p)).$$
Let $y=\sqrt{12n+4}$. As before, $pr_1(5)\leq \frac{2}{5}+\frac{2}{n}$ and $2.5\, pr_1(p)\leq \frac{5}{p}+\frac{5}{n}$ for all $5<p<y$. Then 
$$(1-pr_1(5))(1-2.5\, pr_1(7))\geq \left(1-\frac{2}{5}-\frac{2}{n}\right) \left(1-\frac{5}{7}-\frac{5}{n}\right) =
\left(\frac{3}{5}-\frac{2}{n}\right) \left(\frac{2}{7}-\frac{5}{n}\right)$$
$$=\frac{(3n-10)(2n-35)}{35n^2}=\frac{1}{7}+\frac{n^2-125n+350}{35n^2}>\frac{1}{7}\ \text{for } n>122.$$
Thus we have
$$\log z(n)>-\log\, 7+\sum\limits_{\substack{p\, =\, \text{prime}\\ 7<p<y}} \log \left(1-\frac{5}{p}-\frac{5}{n}\right).$$
Let $u=\frac{5}{p}+\frac{5}{n}$. Then $0<u<\frac{1}{2}$ for all $n>110$ as $p\geq 11$. Now
$$\log(1-u)=-u-\sum\limits_{k\geq 2} \frac{u^k}{k}>-u-\frac{u^2}{2}\, \sum\limits_{k\geq 2} u^{k-2}=-u-\frac{u^2}{2(1-u)}.$$
As $u<\frac{1}{2}$, we have $\frac{1}{2(1-u)}<1$. Thus 
$$\log(1-u)>-u-u^2.$$
Now 
$$\sum\limits_{\substack{p\, =\, \text{prime}\\ 7<p<y}} u=\sum\limits_{\substack{p\, =\, \text{prime}\\ 11\leq p<y}} \left(\frac{5}{p}+\frac{5}{n}\right) =5\, \sum\limits_{\substack{p\, =\, \text{prime}\\ 11\leq p<y}} \frac{1}{p} +\frac{5}{n} \, (\pi(y)-4)$$
and 
$$\sum\limits_{\substack{p\, =\, \text{prime}\\ 7<p<y}} u^2=\sum\limits_{\substack{p\, =\, \text{prime}\\ 11\leq p<y}} \left(\frac{5}{p}+\frac{5}{n}\right)^2 =25\, \sum\limits_{\substack{p\, =\, \text{prime}\\ 11\leq p<y}} \frac{1}{p^2} +\frac{50}{n} \, \sum\limits_{\substack{p\, =\, \text{prime}\\ 11\leq p<y}} \frac{1}{p}+\frac{25}{n^2}\, (\pi(y)-4).$$
We compute the following: 
$$\sum\limits_{\substack{p\, =\, \text{prime}\\ 11\leq p<y}} \frac{1}{p}=\sum\limits_{\substack{p\, =\, \text{prime}\\ p\leq y}} \frac{1}{p} -\left(\frac{1}{2}+\frac{1}{3}+\frac{1}{5}+\frac{1}{7}\right)=\log\, \log\, y +B_1+o(1)-1.176$$
$$<\log\, \log\, y +0.262-1.176=\log\, \log\, y-0.914, (\text{where } B_1=0.2614972128\ \cite{HWR}),$$
$$\sum\limits_{\substack{p\, =\, \text{prime}\\ 11\leq p<y}} \frac{1}{p^2}=\sum\limits_{\substack{p\, =\, \text{prime}\\ p\leq y}} \frac{1}{p^2} -\left(\frac{1}{2^2}+\frac{1}{3^2}+\frac{1}{5^2}+\frac{1}{7^2}\right)<P(2)-0.42<0.46-0.42=0.04,$$
$$\text{ where }P(s)= \sum\limits_{p\,\in\, primes} \frac{1}{p^s} \text{ is the prime zeta function with } P(2)=0.45224742,$$ 
$$\pi(y)-4<\frac{\sqrt{12n+4}}{\log\,\sqrt{12n+4}}<\frac{\sqrt{13n}}{\log\,\sqrt{12n}}<\frac{4\sqrt{n}}{\log\,\sqrt{n}}=\frac{8\sqrt{n}}{\log\, n} \text{ as }\pi(y)\approx\frac{y}{\log\, y}\ \cite{HWR},$$
$$\log\log\, y=\log\log\,\sqrt{12n+4}<\log\log\,\sqrt{16n}=\log\, (\log\, 4+\frac{1}{2}\log\, n)<\log\log\, n\text{ for all } n>16.$$
Then summing up all these we have
$$\log\, z(n)>-\log\, 7+\sum\limits_{\substack{p\, =\, \text{prime}\\ 7<p<y}} \log\,(1-u)>-\log\, 7-\sum\limits_{\substack{p\, =\, \text{prime}\\ 7<p<y}} (u+u^2) >-5\log\log\, n -\epsilon^\prime_n+c$$
where $c=-\log\, 7+5\times 0.914-25\times 0.04=1.624$ and 
$$\epsilon^\prime_n=\frac{5}{n}\,\frac{8\sqrt{n}}{\log\, n}+\frac{50}{n}\,(\log\log\, n-0.914)+\frac{25}{n^2}\,\frac{8\sqrt{n}}{\log\, n}<\frac{50\log\log\, n}{n}+\frac{40}{\sqrt{n}\log\, n}+\frac{200}{n\sqrt{n}\log\, n}=\epsilon_n \text{ (say)}.$$
Now $\epsilon_n$ is a decreasing function for $n>1$. We may choose some value for a (large) $n=k$ and $\epsilon_n<\epsilon_k$ for all $n>k$. Now
$$\log\, z(n)>-5\log\log\, n -\epsilon_k+c$$
which implies 
$$z(n)>\frac{e^{c-\epsilon_k}}{(\log\, n)^5}.$$
Thus $z(n)>\frac{1}{n}$ for all $n$ satisfying 
\begin{equation}\label{nn5}
\frac{n}{(\log\, n)^5}>e^{\epsilon_k-c} \text{ where } c=1.624. 
\end{equation}
We need $N_k>k$ to be the minimum value such that (\ref{nn5}) is satisfied for all $n>N_k$. We observed that $N_k$ is $30990,20744,19552,19140,18964,18840$ for $k=1000,5000,10000,15000,18000,20000$ respectively (see Table \ref{ktable}). We choose $k=18000$ as an optimal value for which $\epsilon_k=0.036776$, $e^{\epsilon_k-c}=0.2045$ and (\ref{nn5}) is satisfied for all $n>18964$. Thus we have $z(n)>\frac{1}{n}$ for all $n>18964$.

\vspace{1em}
\noindent
Similarly, we have 
$z(n)>\frac{2}{n}$ for all $n$ satisfying $\frac{n}{(\log\, n)^5}>2e^{\epsilon_k-c}=0.409$ which is true for all $n>71241$. 
Therefore the result follows as $\tau(n)=1-z(n)$.
\end{proof}

\begin{remark}\label{remgood}
We have exhaustively verified that $1-\tau(n)>\frac{1}{n}$ for all $161\leq n\leq 18964$ and $1-\tau(n)>\frac{2}{n}$ for all $447\leq n\leq 71241$ by machine computing. In fact $\tau(161)=0.993763<0.993789=1-\frac{1}{161}$ and $\tau(447)=0.995523<0.995526=1-\frac{2}{447}$. We note that the degree of $x=12n+4$ in $G(x/2)$ is $n(1-T(n))$. So if $T(n)<1-\frac{1}{n}$, then $n(1-T(n))>1$ which implies that $x$ can be expressed as the sum of two positive integers which are odd primes or $1$. Similarly, if $T(n)<1-\frac{2}{n}$, then $n(1-T(n))>2$ which implies that $x$ can be expressed as the sum of two odd primes. 
\end{remark}

\begin{theorem}\label{tgood}
If the set of divisibility events for $x=12n+4$ {\em ($n\in\Nat$, $n>447$)} is nearly independent, then $x$ can be expressed as the sum of two odd primes.
\end{theorem}

\begin{proof}
The proof follows from Lemmas \ref{lem1good}, \ref{lem2good} and Remark \ref{remgood}.
\end{proof}

\begin{table}[t]
{\footnotesize $$\begin{array}{|c|l|l|l|l|l|l|l|l|l|l|}
\hline
n & 10 & 10^2 & 10^3 & 10^4 & 10^5 & 10^6 & 10^7 & 10^8 & 10^9 & 10^{10}\\
\hline
T(n) & 0.5 & 0.73 & 0.888 & 0.9251 & 0.94666 & 0.965814 & 0.974293 & 0.975922 & 0.983884 & 0.986781\\
\hline
\tau(n) & 0.86875 & 0.957164 & 0.99596 & 0.998257 & 0.999172 & 0.999718 & 0.999855 & 0.999817 & 0.999952 & 0.99997\\
\hline
\end{array}$$}
\caption{Comparison table for $T(n)$ and $\tau(n)$ for $x=12n+4$.}\label{tabtnetafn}
\end{table}

\vspace{1em}
\noindent
The discussion we made so far in this section for a particular form of even integers, namely, $x=12n+4$, ($n\in\Nat$) for clear understanding that avoids technical rigor if we did it generally for all possible forms at a time. In view of Theorems \ref{thm24810}, \ref{thm612}, Table \ref{tablamx} and Tables \ref{tabtetan0} - \ref{tabtetan8}, we see that $\mathscr{A}(x)$ is a set of nearly independent events for any $x\leq 2\times 10^7$ and for $x=12n+k$, $n=10^i$ ($1\leq i\leq 10$), $k=0,2,4,6,8,10$, i.e., in all other forms also. Following the above arguments one can prove Theorem \ref{tgood} for large even integers, in general, and we may conclude the following:

\begin{theorem}\label{tgoodgen}
If the set of divisibility events for a large even positive integer is nearly independent, then $x$ can be expressed as the sum of two odd primes.
\end{theorem}

\noindent
Now we wish to study the possible bounds for $\lambda_0=\max\set{\lambda(A,B)}$. By our observations, we assumed $\lambda_0<2.5$. It may not be easy to prove. Moreover, the global constant upper bound $\lambda_0$ cannot be $3.5$ or more as we need $\frac{2\lambda_0}{7}<1$ to prove Theorem \ref{tgood} (the other term $\frac{2\lambda_0}{n}$ would be small for large $n$). Thus in the following we wish to estimate the possible upper bound for $\lambda$ which may not be constant but dynamic.

\vspace{1em}
\noindent
Let $\mathscr{A}=\set{A_1,A_2,\ldots,A_r}$ be a set of divisibility events. Let $A=\bigcup\limits_{j=1}^{i-1} A_j$ and $B=A_i$ for some $i$ where $2\leq i\leq r$. Let $\lambda=\lambda(A,B)=\frac{P(B)-P(A\cap B)}{P(B)-P(A)P(B)}$. 

\begin{lemma}\label{kpb1}
Suppose $P(A)<1$. Then $P(A\cup B)<1$ if and only if $\lambda P(B)<1$.
\end{lemma}

\begin{proof}
We have $P(A\cup B)=P(A)+P(B)-P(A\cap B)<1\Longleftrightarrow P(B)-P(A\cap B)<1-P(A)\Longleftrightarrow \lambda=\frac{P(B)-P(A\cap B)}{P(B)-P(A)P(B)}<\frac{1}{P(B)}\Longleftrightarrow \lambda P(B)<1$.
\end{proof}

\noindent
In the following, let us estimate an upper bound of $\lambda$. We note that 
$$\frac{P(B)-P(A\cap B)}{P(B)-P(A)P(B)}\leq \frac{1}{1-P(A)}$$
for any such $A$ and $B$ as $P(A\cap B)\geq 0$. Thus if we can estimate $\frac{1}{1-P(A)}$, we can get an upper bound of $\lambda$. Now suppose $B=A_{p_i}$ for some prime $p_i$ and $A=\bigcup\limits_{j=1}^{i-1} A_{p_j}$ (see Definition \ref{defap}). Let us estimate $P(A_{p_j})=pr_3(p_j)$ where $pr_3(p_j)=\frac{1}{p_j}$ if $\Div{p}{x}$ and $=\frac{2}{p_j}$ if $\nDiv{p}{x}$ by (\ref{eqnetaap}). Then we have the following approximation: 
$$1-P(A)\approx\prod\limits_{\Div{p}{x},\, 3<p<p_i} \left(1-\frac{1}{p}\right)\, \prod\limits_{\nDiv{p}{x},\, 3<p<p_i} \left(1-\frac{2}{p}\right).$$
Now $\frac{1}{p}<\frac{2}{p}\Longrightarrow 1-\frac{2}{p}<1-\frac{1}{p}\Longrightarrow \frac{1}{1-\frac{1}{p}}<\frac{1}{1-\frac{2}{p}}$. Thus we have 
$$\frac{1}{1-P(A)}<\frac{1}{\prod\limits_{3<p<p_i} \left(1-\frac{2}{p}\right)}=\frac{1}{3\prod\limits_{2<p<p_i} \left(1-\frac{2}{p}\right)}\sim\frac{(\log\, p_i)^2}{3e^{-\beta}} = 0.400434\, (\log\, p_i)^2<\frac{(\log\, p_i)^2}{2}$$
(by Theorem \ref{2bypform}), where $\beta=0.183407$. 

\begin{definition}\label{defblncd}
Let $\mathscr{A}=\set{A_1,A_2,\ldots,A_r}$ be a collection of $r$ events. Let $\lambda (A,B)=\frac{P(B)-P(A\cap B)}{P(B)-P(A)P(B)}$, $\lambda_i=\lambda (\bigcup\limits_{j=1}^{i-1} A_j,A_i)$ for $i=2,3,\ldots,r$. Then $\mathscr{A}$ is called a set of {\em balanced} events if $\lambda_i<\frac{(\log\, p_i)^2}{2}$, where $p_i$ is the $(i+2)^{\text{th}}$ prime, for all $i=2,3,\ldots,r$.
\end{definition}

\noindent
Note that $\frac{(\log\, p)^2}{2}>4$ for any prime $p\geq 17$. Thus it is a weaker condition than the former one.

\begin{theorem}\label{thmblncd}
If the set of divisibility events for a large even positive integer $x$ is balanced, then $T(n)<1$. 
\end{theorem}

\begin{proof}
Following the proof of Theorem \ref{tgood}, we have $T(n)=P(\bigcup\limits_{j=1}^{r_x} A_{p_j})$, where $p_j$ denotes the $(j+2)^{\text{th}}$ prime. We consider $n$ to be large enough so that we may ignore $\frac{1}{n}$ and other small terms and assume $P(A_{p_j})=pr_3(p_j)$ for $j=1,2,\ldots, r_x$ (see (\ref{eqnetaap})). We prove $T(n)<1$ by induction.

\vspace{1em}
\noindent
As the basic step we have $P(A_5)=pr_3(5)\leq \frac{2}{5}<1$ and 
$$P(A_5\cup A_7)\leq P(A_5)+P(A_7)\leq \frac{2}{5}+\frac{2}{7}=0.685714<1.$$
Also note that $\lambda_2 P(A_7)<\frac{(\log\, 7)^2}{2}\,\frac{2}{7}=0.540938<1$ by Definition \ref{defblncd}, where $\lambda_2=\lambda(A_5,A_7)$. Thus by Lemma \ref{kpb1}, we have $P(A_5\cup A_7)<1$.

\vspace{1em}
\noindent
Now suppose for some $i>1$, $A=\bigcup\limits_{j=1}^{i-1} A_j$, $B=A_{p_i}$ and $P(A)<1$.\\[1em]
By Definition \ref{defblncd}, $\lambda(A,B)=\lambda_i<\frac{(\log\, p_i)^2}{2}$. 
Then 
$$\lambda_i P(B)< \frac{(\log\, p_i)^2}{2}\,\frac{2}{p_i}=\frac{(\log\, p_i)^2}{p_i}<1 \text{ as } p_i>5.$$
Then by Lemma \ref{kpb1}, we have $P(\bigcup\limits_{j=1}^{i} A_j)=P(A\cup B)<1$. Thus, by induction, we have\\
 $T(n)=P(\bigcup\limits_{j=1}^{r_x} A_{p_j})<1$ as required.
\end{proof}

\begin{remark}
It is interesting to note that the upper bound of $\frac{1}{1-P(A)}$ estimated above (before Definition \ref{defblncd}) as $\frac{(\log\, p_i)^2}{2}$ incidentally proposes a closer upper bound of $T(n)$ than $\tau(n)$, namely, 
$$T(n)\leq 1-\frac{2}{(\log\, p_i)^2}=\tau^\prime (n) \text{ (say)},$$
where $p_i$ is the highest prime less than $\sqrt{x}$ (for $x=12n+4$, see Table \ref{newubound}).\\
 Also note that $\tau^\prime (n)=1-\frac{2}{(\log\, p_i)^2}<1-\frac{2}{(\log\, \sqrt{12n+4})^2}<1-\frac{2}{n}$ for all $n>3$.
\end{remark}

\begin{table}[ht]
{\footnotesize $$\begin{array}{|c|l|l|l|l|l|l|l|l|l|l|}
\hline
n & 10 & 10^2 & 10^3 & 10^4 & 10^5 & 10^6 & 10^7 & 10^8 & 10^9 & 10^{10}\\
\hline
T(n) & 0.5 & 0.73 & 0.888 & 0.9251 & 0.94666 & 0.965814 & 0.974293 & 0.975922 & 0.983884 & 0.986781\\
\hline
\tau^\prime (n) & 0.652168 & 0.830397 & 0.909127 & 0.940957 & 0.959145 & 0.969889 & 0.976881 & 0.981694 & 0.985147 & 0.987707\\
\hline
\frac{\tau^\prime (n)}{T(n)} & 1.30434 & 1.13753 & 1.02379 & 1.01714 & 1.01319 & 1.00422 & 1.00266 & 1.00591 & 1.00128 & 1.00094 \\
\hline
\end{array}$$}
\caption{Comparison table for $T(n)$ and $\tau^\prime (n)$ for $x=12n+4$.}\label{newubound}
\end{table}

\section{Conclusion}

In this paper, we computed degrees of a vertex (positive even integer) $x$ in the near Goldbach graph $G(x/2)$. We first locate some desired subset of vertices in which all neighbors of $x$ may lie (except possibly one when $x-3$ is a prime number). Then we found two exact formulas for degree of any $x$ in $G(x/2)$ which gives exact formulas for counting the number of ways that an even positive integer can be expressed as a sum of two odd primes. In order to calculate values for large $x$, we require approximations. We compared approximated values with the exact values for a large number of even integers to observe close similarity between them. Moreover we obtain a closed form as a constant multiple of $\frac{x}{(\log\, x)^2}$ that approximates the degree of (large) $x$ in $G(x/2)$, which is to close to Hardy-Littlewood conjecture. But the major problem is to show analytically that the exact value of $T(n)=\frac{f(n)}{n}<1$ or to be precise less than $1-\frac{2}{n}$, where $1-T(n)$ is the probability of a vertex (in the desired set) to be a neighbor of $x$. We assumed some conditions on the basis of our observations. With one of these assumptions we could be able to produce a function $\tau(n)$ and prove that $0<T(n)<\tau(n)<1-\frac{2}{n}$ for large $x$ which implies $x$ can be expressed as a sum of two odd primes. The other one leads to simply $0<T(n)<1$. All we need is to prove these conditions in future to complete the solution.

\vspace{2em}
\noindent
{\bf Acknowledgement}

\vspace{1em}
\noindent
This research is partially supported by Science \& Engineering Research Board (SERB) (presently Anusandhan National Research Foundation (ANRF)), Department of Science and Technology, Government of India under MATRICS Scheme (MTR/2022/000050 dated 29.12.2022). We express our gratitude to Professor Buddhadeb Sau, Department of Mathematics, Jadavpur University, India, for providing access to a high-performance server that enabled computations with large integers. We also extend our sincere thanks to Dr.~Arghya Datta of D$\acute{\text{e}}$partement de math$\acute{\text{e}}$matiques et de statistique, Universit$\acute{\text{e}}$ de Montr$\acute{\text{e}}$al, Canada and Mr. Biswanath Samanta of Indian Institute of Science Education and Research, Pune, India for their valuable suggestions and insightful discussions, which have significantly improved this paper.

\newpage

\section{Appendix I}

\begin{table}[h]
$$\begin{array}{|l|l|l|l|l|l|l|l|}
\hline
n & T(n) & \tau_1(n) & \tau_2(n) & T(n)/\tau_2(n) & \tau(n) & \lambda_0 & \text{degree} \\
\hline 
10 &  0.4 & 0.4 & 0.605 & 0.66116 & 0.7 & 1 & 12 \\
\hline
10^2 & 0.73 & 0.739975 & 0.786916 & 0.92767 & 0.974453 & 1.13636 & 54\\
\hline
10^3 & 0.849 & 0.85605 & 0.864145 & 0.98247 & 0.994627 & 1.34454 & 302\\
\hline
10^4 & 0.9055 & 0.903006 & 0.904325 & 1.0013 & 0.998027 & 1.34269 & 1890\\
\hline
10^5 & 0.93661 & 0.932644 & 0.932891 & 1.00399 & 0.99921 & 1.45267 & 12678\\
\hline
10^6 & 0.954561 & 0.950011 & 0.95006 & 1.00474 & 0.999625 & 1.49085 & 90877\\
\hline
10^7 & 0.965758 & 0.961591 & 0.961601 & 1.00432 & 0.999806 & 1.49346 & 684832\\
\hline 
10^8 & 0.97324 & 0.969561 & 0.969563 & 1.00379 & 0.999892 & 1.51027 & 5352052\\
\hline
10^9 & 0.978518 & 0.975285 & 0.975286 & 1.00331 & 0.999936 & 1.52426 & 42963384\\
\hline
10^{10} & 0.982375 & 0.979541 & 0.979541 & 1.00289 & 0.99996 & 1.53962 & 352503092\\
\hline
\end{array}$$
\caption{Values of $T(n)$, $\tau_1(n),\tau_2(n),\tau(n)$, $\lambda_0$ and degree$(x)$ for $x=12n$.}\label{tabtetan0}
\end{table}

\vspace{1em}

\begin{table}[h]
$$\begin{array}{|l|l|l|l|l|l|l|l|}
\hline
n & T(n) & \tau_1(n) & \tau_2(n) & T(n)/\tau_2(n) & \tau(n) & \lambda_0 & \text{degree}\\
\hline 
10 &  0.7 & 0.616 & 0.848104 & 0.82537 & 0.85 & 1.66667 & 3\\
\hline
10^2 & 0.81 & 0.806024 & 0.843382 & 0.96042 & 0.981125 & 1.53846 & 19\\
\hline
10^3 & 0.885 & 0.88502 & 0.891399 & 0.99282 & 0.995157 & 1.40292 & 115\\
\hline
10^4 & 0.9274 & 0.924583 & 0.925596 & 1.00195 & 0.998368 & 1.65856 & 726\\
\hline
10^5 & 0.94756 & 0.943964 & 0.944171 & 1.00359 & 0.999204 & 1.60389 & 5245\\
\hline
10^6 & 0.959249 & 0.95501 & 0.955054 & 1.00439 & 0.999368 & 1.54613 & 40751\\
\hline
10^7 & 0.974162 & 0.971 & 0.971008 & 1.00325 & 0.999852 & 1.56817 & 258382\\
\hline 
10^8 & 0.979928 & 0.977171 & 0.977172 & 1.00282 & 0.999919 & 1.54577 & 2007225\\
\hline
10^9 & 0.983888 & 0.981464 & 0.981464 & 1.00247 & 0.999952 & 1.57394 & 16111594\\
\hline
10^{10} & 0.986277 & 0.984072 & 0.984072 & 1.00224 & 0.999967 & 1.56963 & 137230841\\
\hline
\end{array}$$
\caption{Values of $T(n)$, $\tau_1(n),\tau_2(n),\tau(n)$, $\lambda_0$ and degree$(x)$ for $x=12n+2$.}\label{tabtetan2}
\end{table}

\newpage

\begin{table}[h]
$$\begin{array}{|l|l|l|l|l|l|l|l|}
\hline
n & T(n) & \tau_1(n) & \tau_2(n) & T(n)/\tau_2(n) & \tau(n) & \lambda_0 & \text{degree}\\
\hline 
10 &  0.52381 & 0.54519 & 0.749627 & 0.69876 & 0.744169 & 1.07692 & 10\\
\hline
10^2 & 0.78607 & 0.803124 & 0.8405 & 0.93524 & 0.980285 & 1.27909 & 43\\
\hline
10^3 & 0.877061 & 0.882841 & 0.889174 & 0.98638 & 0.994963 & 1.36394 & 246\\
\hline
10^4 & 0.927854 & 0.925321 & 0.926318 & 1.00166 & 0.998417 & 1.46824 & 1443\\
\hline
10^5 & 0.95236 & 0.949045 & 0.94923 & 1.0033 & 0.999394 & 1.79442 & 9528\\
\hline
10^6 & 0.965979 & 0.962508 & 0.962545 & 1.00357 & 0.999719 & 1.5747 & 68042\\
\hline
10^7 & 0.969165 & 0.965432 & 0.965441 & 1.00386 & 0.999673 & 1.53195 & 616691\\
\hline 
10^8 & 0.979939 & 0.977171 & 0.977172 & 1.00283 & 0.999919 & 1.5523 & 4012107 \\
\hline
10^9 & 0.98389 & 0.981464 & 0.981464 & 1.00247 & 0.999952 & 1.53503 & 32219199 \\
\hline
10^{10} & 0.985993 & 0.983743 & 0.983743 & 1.00229 & 0.999964 & 1.55218 & 280130367\\
\hline
\end{array}$$
\caption{Values of $T(n)$, $\tau_1(n),\tau_2(n),\tau(n)$, $\lambda_0$ and degree$(x)$ for $x=12n+6$.}\label{tabtetan6}
\end{table}

\vspace{1em}

\begin{table}[h]
$$\begin{array}{|l|l|l|l|l|l|l|l|}
\hline
n & T(n) & \tau_1(n) & \tau_2(n) & T(n)/\tau_2(n) & \tau(n) & \lambda_0 & \text{degree}\\
\hline 
10 &  0.636364 & 0.574005 & 0.83411 & 0.76293 & 0.810669 & 1.57143 & 4\\
\hline
10^2 & 0.80198 & 0.80452 & 0.84311 & 0.95122 & 0.980702 & 1.35753 & 20\\
\hline
10^3 & 0.883117 & 0.884274 & 0.890916 & 0.99125 & 0.9951 & 1.67251 & 117\\
\hline
10^4 & 0.916708 & 0.912725 & 0.913912 & 1.00306 & 0.996674 & 1.61977 & 833\\
\hline
10^5 & 0.95234 & 0.949481 & 0.94967 & 1.00281 & 0.999407 & 1.53231 & 4766\\
\hline
10^6 & 0.965978 & 0.962427 & 0.962464 & 1.00365 & 0.999717 & 1.62772 & 34022\\
\hline
10^7 & 0.971657 & 0.968213 & 0.968222 & 1.00355 & 0.999807 & 1.60554 & 283436\\
\hline 
10^8 & 0.979917 & 0.977171 & 0.977172 & 1.00281 & 0.999919 & 1.58452 & 2008273\\
\hline
10^9 & 0.983887 & 0.981464 & 0.981464 & 1.00247 & 0.999952 & 1.54933 & 16112912\\
\hline
10^{10} & 0.984137 & 0.981587 & 0.981587 & 1.0026 & 0.999932 & 1.55399 & 158634731\\
\hline
\end{array}$$
\caption{Values of $T(n)$, $\tau_1(n),\tau_2(n),\tau(n)$, $\lambda_0$ and degree$(x)$ for $x=12n+8$.}\label{tabtetan8}
\end{table}

\newpage

\begin{table}[h]
$$\begin{array}{|l|l|l|l|l|l|l|l|}
\hline
n & T(n) & \tau_1(n) & \tau_2(n) & T(n)/\tau_2(n) & \tau(n) & \lambda_0 & \text{degree}\\
\hline 
10 &  0.4 & 0.496 & 0.741455 & 0.53948 & 0.85 & 1 & 7\\
\hline
10^2 & 0.68 & 0.709893 & 0.76328 & 0.890892 & 0.963464 & 1.25 & 32\\
\hline
10^3 & 0.849 & 0.856183 & 0.864162 & 0.982455 & 0.994669 & 1.46452 & 152\\
\hline
10^4 & 0.8957 & 0.892232 & 0.893689 & 1.00225 & 0.997206 & 1.43747 & 1043\\
\hline
10^5 & 0.92078 & 0.915304 & 0.915606 & 1.00565 & 0.997985 & 1.55161 & 7923\\
\hline
10^6 & 0.947312 & 0.942104 & 0.94216 & 1.00547 & 0.999408 & 1.53846 & 52688\\
\hline
10^7 & 0.960428 & 0.955617 & 0.955629 & 1.00502 & 0.9997 & 1.89358 & 395717\\
\hline 
10^8 & 0.970145 & 0.966056 & 0.966058 & 1.00423 & 0.999845 & 1.5451 & 2985520\\
\hline
10^9 & 0.978489 & 0.975258 & 0.975258 & 1.00331 & 0.999935 & 1.55233 & 21510901\\
\hline
10^{10} & 0.979089 & 0.97573 & 0.97573 & 1.00344 & 0.999931 & 2.10691 & 209105088\\
\hline
\end{array}$$
\caption{Values of $T(n)$, $\tau_1(n),\tau_2(n),\tau(n)$, $\lambda_0$ and degree$(x)$ for $x=12n+10$.}\label{tabtetan10}
\end{table}

\begin{table}[h]
\centering
\begin{tabular}{|c|c|c|c|c|}
\hline
 & & & & Minimum $N_k$ solving \\
$k$ & $\epsilon_k$ & $C_k = e^{\,c-\epsilon_k}$ & $1/C_k$ & $\displaystyle \frac{n}{(\log n)^5} > \frac{1}{C_k}$ \\
\hline
$1000$  & $0.294960$ & $3.0323$ & $0.3297$ & $30990$ \\
$3000$  & $0.126037$ & $4.4726$ & $0.2236$ & $22695$ \\
$5000$  & $0.084893$ & $4.7998$ & $0.2083$ & $20744$ \\
$10000$ & $0.056001$ & $5.0286$ & $0.1989$ & $19552$ \\
$15000$ & $0.042480$ & $5.1644$ & $0.1936$ & $19140$ \\
$18000$ & $0.036776$ & $4.8902$ & $0.2045$ & $18964$ \\
$20000$ & $0.033789$ & $5.2638$ & $0.1899$ & $18840$ \\
$25000$ & $0.028092$ & $5.3407$ & $0.1873$ & $18694$ \\
\hline
\end{tabular}
\caption{Numerical values of $\epsilon_k$, $C_k = e^{\,c-\epsilon_k}$, $1/C_k$, 
and the corresponding analytic threshold $N_k$.}\label{ktable}
\end{table}

\newpage

\section{Appendix II}
\noindent
{\bf Another justification for Hardy-Littlewood Formula}

\noindent
In the following we provide a short heuristic justification of Hardy-Littlewood formula which is different from Remark \ref{rem:HL}. As before, it is not to be considered as a rigorous proof.

\vspace{1em}
\noindent
For an even integer greater than $5$, we wish to count the number of cases for which $x=p+(x-p)$, where both $p$ and $x-p$ are odd primes. Let $D(x)$ be the number of ways that $x$ is represented as the sum of two odd primes (where $p+q$ and $q+p$ are considered as different representations), then the following is a conjecture by Hardy and Littlewood:
$$D(x)\sim 2 \frac{C(x) x}{(\log\, x)^2},\ \text{where } C(x)=\kappa(x)\, \Pi_2,\ \kappa(x)=\prod\limits_{\Div{p}{x},\, p>2} \frac{p-1}{p-2}\text{ and }\Pi_2= \prod\limits_{p>2} \left( 1-\frac{1}{(p-1)^2}\right).$$

\noindent
Note that the product $\Pi_2$ is convergent and the constant $0.660161858...$ is known as the {\em twin-prime constant}. 

Let $x$ be an even positive integer ($>5$) such that $x=y+(x-y)$. Let $\pi(x)$ denote the number of primes less than or equal to $x$. Now the probability that $y$ is an odd prime is $\frac{\pi(x)-1}{x}$ and that $x-y$ is an odd prime is also $\frac{\pi(x)-1}{x}$. So the probability that both of them are odd primes would be $\left(\frac{\pi(x)-1}{x}\right)^2$ if the choices would be independent. But they are not independent as they are related by $y+(x-y)=x$.

\vspace{1em}
\noindent
Now $y$ and $(x-y)$ are odd prime numbers if they are not divisible by earlier odd primes. So, in the case, where a prime $p\nmid x$, we have to exclude two cases from its residue list, namely, $\nMod{y}{0}{p}$ and $\nMod{y}{x}{p}$. Thus the actual density is $\frac{p-2}{p}$ for each block of $p$ positive integers before $x$, whereas for independent choice of $y$ and $x-y$ it would be $\left(\frac{p-1}{p}\right)^2$. So the correction factor for each prime $p$ is $$\frac{\frac{p-2}{p}}{\left(\frac{p-1}{p}\right)^2}\ =\ \frac{p(p-2)}{(p-1)^2}.$$
This leads to the product (for all primes $2<p<\sqrt{x}$, $\nDiv{p}{x}$):
$$\prod\limits_{p\nmid x,\, 2<p<\sqrt{x}} \frac{p(p-2)}{(p-1)^2}=\frac{\prod\limits_{p\nmid x,\,2<p<\sqrt{x}} \left(1-\frac{2}{p}\right)}{\prod\limits_{p\nmid x,\,2<p<\sqrt{x}} \left(1-\frac{1}{p}\right)^2}.$$

\vspace{1em}
\noindent
Similarly, for $p\mid x$, the actual density is $\frac{p-1}{p}$ whereas the independent choice yields $\left(\frac{p-1}{p}\right)^2$. Thus the correction factor in this case is $\frac{p}{p-1}$ which gives the product 
$$\frac{1}{\prod\limits_{p\mid x,\,2<p<\sqrt{x}} \left(1-\frac{1}{p}\right)}.$$

Therefore the required number is 
$$x\,\left(\frac{\pi(x)-1}{x}\right)^2\ \frac{1}{\prod\limits_{p\mid x,\,2<p<\sqrt{x}} \left(1-\frac{1}{p}\right)}\ \frac{\prod\limits_{p\nmid x,\,2<p<\sqrt{x}} \left(1-\frac{2}{p}\right)}{\prod\limits_{p\nmid x,\,2<p<\sqrt{x}} \left(1-\frac{1}{p}\right)^2}$$
$\displaystyle{=\ \frac{(\pi(x)-1)^2}{x}\, \prod\limits_{p\mid x,\,2<p<\sqrt{x}} \left(1-\frac{1}{p}\right)\, \frac{\prod\limits_{p\nmid x,\,2<p<\sqrt{x}} \left(1-\frac{2}{p}\right)}{\, \prod\limits_{p\mid x,\,2<p<\sqrt{x}} \left(1-\frac{1}{p}\right)^2\, \prod\limits_{p\nmid x,\,2<p<\sqrt{x}} \left(1-\frac{1}{p}\right)^2}}$\\[1em]
$\displaystyle{=\ \frac{(\pi(x)-1)^2}{x}\, \frac{\prod\limits_{p\mid x,\,2<p<\sqrt{x}} \left(1-\frac{1}{p}\right)}{\prod\limits_{p\mid x,\,2<p<\sqrt{x}} \left(1-\frac{2}{p}\right)}\, \frac{\, \prod\limits_{p\mid x,\,2<p<\sqrt{x}} \left(1-\frac{2}{p}\right)\, \prod\limits_{p\nmid x,\, 2<p<\sqrt{x}} \left(1-\frac{2}{p}\right)}{\, \prod\limits_{2<p<\sqrt{x}} \left(1-\frac{1}{p}\right)^2}}$\\[1em]
$\displaystyle{=\ \frac{(\pi(x)-1)^2}{x}\, \prod\limits_{p\mid x,\,2<p<\sqrt{x}} \frac{p-1}{p-2}\,  \prod\limits_{2<p<\sqrt{x}} \frac{p(p-2)}{(p-1)^2}}$\\[1em]
$\displaystyle{=\ \frac{(\pi(x)-1)^2}{x}\, \prod\limits_{p\mid x,\,2<p<\sqrt{x}} \frac{p-1}{p-2}\, \prod\limits_{2<p<\sqrt{x}} \left(1-\frac{1}{(p-1)^2}\right)}$.

\vspace{1em}
\noindent
Now we notice that if there are $p\mid x$ such that $p>\sqrt{x}$. Then it can be at most one as then $\frac{x}{p}<\sqrt{x}$. Thus $$\prod\limits_{p\mid x,\,p\geq \sqrt{x}} \frac{p-1}{p-2} = \prod\limits_{p\mid x,\,p>\sqrt{x}} \frac{p-1}{p-2}\rightarrow 1\text{ as } x\rightarrow\infty.$$
Thus $\displaystyle{\prod\limits_{p\mid x,\,2<p<\sqrt{x}} \frac{p-1}{p-2}\sim \prod\limits_{p\mid x,\, p>2} \frac{p-1}{p-2}=\kappa(x)}$.

\vspace{1em}
\noindent
Moreover, $\Pi_2=\displaystyle{\prod\limits_{p>2} \left(1-\frac{1}{(p-1)^2}\right)}$ is a convergent series. Thus $\displaystyle{\prod\limits_{2<p<\sqrt{x}} \left(1-\frac{1}{(p-1)^2}\right)}\rightarrow \Pi_2$ as $x\rightarrow\infty$.

\vspace{1em}
\noindent
So, we have, 
$$\frac{1}{2} D(x)\sim \frac{C(x) (\pi(x)-1)^2}{x},$$ 
where $C(x)=\prod\limits_{\Div{p}{x},\, p>2} \frac{p-1}{p-2}\ \prod\limits_{p>2} \left( 1-\frac{1}{(p-1)^2}\right)\ =\ \kappa(x)\, \Pi_2$.

\vspace{1em}
Finally, since $\pi(x)\sim \frac{x}{\log x}$ by Prime number theorem, we have
$$D(x)\sim 2\frac{C(x)\, x}{(\log x)^2}$$
which is exactly the formula conjectured by Hardy and Littlewood. 
\end{document}